\documentclass[11pt]{article}
\usepackage{color}
\definecolor{aleacolour}{rgb}{0.09,0.32,0.44} 
\usepackage[pdfborder={0 0 0},pdfborderstyle={},colorlinks=true,linkcolor=aleacolour,citecolor=aleacolour,urlcolor=blue]{hyperref}
\usepackage{amsthm,amsmath,amssymb,amsfonts}
\usepackage{enumerate}
\usepackage{cases}
\usepackage[top=1in,bottom=1in,left=1in,right=1in]{geometry}
\usepackage{graphicx}

\usepackage{bbm}

\newcommand{\zz}{\mathbb{Z}}
\newcommand{\ee}{\mathbb{E}}

\def\({\left(}
\def\){\right)}

\newcommand{\dd}{{\delta}}

\newcommand{\sstar}{{s^*}}
 
\newcommand{\integer}[1]{\left\lfloor #1 \right\rfloor} 
\newcommand{\ceiling}[1]{\left\lceil #1 \right\rceil} 
\newcommand\given[1][]{\:#1\vert\:}

\newcommand{\yb}[1]{Z_{#1}^B} 
\newcommand{\ys}{Z_S} 
\newcommand{\ysn}{Z^n_S} 
\newcommand{\ysq}[1]{Z_{S_{#1}}}

\newcommand{\ssett}{\mathcal{A}}
\newcommand{\sset}{\mathcal{AP}}
\newcommand{\ssetb}{\overline{\mathcal{AP}}}

\newcommand{\fv}{\tilde f}
\newcommand{\bench}[1]{h^*_{#1}}

\newcommand{\sm}{\sigma_\mu}

\newcommand{\tb}{t^{\text{b}}}
\newcommand{\tg}{t^{\text{g}}}
\newcommand{\Ub}{U_{\text{b}}}
\newcommand{\Ug}{U_{\text{g}}}
\newcommand{\x}{x} 

\theoremstyle{plain}
\newtheorem{theorem}{Theorem}

\newtheorem{lemma}[theorem]{Lemma}
\newtheorem{claim}[theorem]{Claim}
\newtheorem{corollary}[theorem]{Corollary}
\newtheorem{conjecture}[theorem]{Conjecture}

\theoremstyle{definition}
\newtheorem{definition}[theorem]{Definition}

\theoremstyle{remark}

\theoremstyle{property}

\title{ On the asymptotic confirmation of the Faudree-Lehel Conjecture for general graphs}

\author{
Jakub Przyby{\l}o\thanks{AGH University of Science and Technology, Faculty of Applied Mathematics, al. A. Mickiewicza 30, 30-059 Krakow, Poland. 
 Email: \href{mailto:jakubprz@agh.edu.pl} {\nolinkurl{jakubprz@agh.edu.pl}}.}
\and 
Fan Wei\thanks{Department of Mathematics, Princeton University, Princeton, NJ 08540, USA. Research supported by NSF Award DMS-1953958. Email: \href{mailto:fanw@princeton.edu} {\nolinkurl{fanw@princeton.edu}}.}
}
\date{}

\begin{document}
\maketitle  

\begin{abstract}
Given a simple graph $G$, the {\it irregularity strength} of $G$, denoted by $s(G)$, is the least positive integer $k$ such that there is a weight assignment on edges $f: E(G) \to \{1,2,\dots, k\}$  
attributing distinct weighted degrees: $\fv(v):= \sum_{u: \{u,v\}\in E(G)} f(\{u,v\})$ 
to all vertices $v\in V(G)$.
It is straightforward that $s(G) \geq n/d$ for every $d$-regular graph $G$ on $n$ vertices with $d>1$. In 1987, Faudree and Lehel conjectured in turn that there is an absolute constant $c$ such that $s(G) \leq n/d + c$ for all such graphs. 
Even though the conjecture has remained open in almost all relevant cases, it is more generally believed that 
there exists a universal constant $c$ such that $s(G) \leq n/\dd + c$ for every graph $G$ on $n$ vertices with minimum degree $\dd\geq 1$ which does not contain an isolated edge.
In this paper we confirm that the generalized Faudree-Lehel Conjecture holds 
for graphs with 
$\dd\geq n^\beta$ where $\beta$ is any fixed constant larger than $0.8$.
Furthermore, we confirm that 
the conjecture 
holds in general asymptotically.
That is we prove that for any $\varepsilon\in(0,0.25)$ there exist absolute constants $c_1, c_2$  such that for all  graphs $G$ on $n$ vertices with minimum degree 
$\dd\geq 1$ and without isolated edges, 
$s(G) \leq \frac{n}{\dd}(1+\frac{c_1}{\dd^{\varepsilon}})+c_2$,
thus extending in various aspects and strengthening a recent result of Przyby{\l}o, who showed that 
$s(G) \leq \frac{n}{d}(1+ \frac{1}{\ln^{\epsilon/19}n})=\frac{n}{d}(1+o(1))$ 
for $d$-regular graphs with 
$d\in [\ln^{1+\epsilon} n, n/\ln^{\epsilon}n]$, and improving an earlier general upper bound: $s(G)< 6\frac{n}{\dd}+6$ of Kalkowski, Karo\'nski and Pfender.
\end{abstract}

\section{Introduction}
Let $G$ be a simple graph. For a positive integer $k$, an edge-weighting function 
$f: E(G)\to \{1,2,\dots, k\}$ is called a \emph{$k$-irregular assignment} for $G$ if the \emph{weighted degrees}, denoted by $\fv(v) = \sum_{u\in N(v)} f(\{v,u\})$ are pairwise distinct for all $v \in V(G)$. The \emph{irregularity strength} of $G$, denoted $s(G)$, is the least positive integer $k$, if exists, such that there is a $k$-irregular  assignment  
for $G$; we set $s(G)=\infty$ for the remaining graphs.  It is easy to see that $s(G) < \infty$ if and only if $G$ has no isolated edges and at most one isolated vertex. 

The irregularity strength was first introduced by Chartrand, 
Jacobson, Lehel,  Oellermann,  Ruiz, and Saba~\cite{CJLORS}, 
in particular in reference to research on irregular graphs~\cite{HighlyIrregular,k-PathIrregular,ChartrandErdosOellermann}, 
as $s(G)$ may naturally be alternatively set down 
as the least $k$ such that one may produce an irregular multigraph, i.e. a multigraph with pairwise distinct degrees, of a given simple graph $G$ by blowing each edge $e$ of $G$ to at most $k$ copies of $e$.
In general it is known that $s(G) \leq n -1$ for any graph $G$ of order $n$ and with finite irregularity strength 
which is not a triangle, 
as proved by Aigner and Triesch~\cite{AT} and Nierhoff~\cite{Nsingle}. Though the family of stars witnesses the tightness of this upper bound, it can be greatly improved for graphs with larger minimum degree.
In particular, already 
Faudree and Lehel~\cite{FL} showed that $s(G) \leq \ceiling{n/2}+9$ for every $d$-regular graph with $n, d\geq 2$.
This was however still far from the expected optimal upper bound.
By a simple counting argument, it is easy to see that 
\[s(G) \geq \ceiling{\frac{n+d+1}{d}}.\] 
This lower bound motivated Faudree and Lehel~\cite{FL} to conjecture in 1987 that the value $n/d$ is close to optimal.
In fact, this conjecture was first posed by Jacobson, as mentioned in~\cite{L-survey}. 
\begin{conjecture}[Faudree-Lehel Conjecture \cite{FL}]\label{conj:main}
There is a constant $C>0$ such that for all $d$-regular graphs $G$ with $n$ vertices and $d\geq 2$, \[s(G) \leq \frac{n}{d}+C.\] 
\end{conjecture}
It is moreover believed that the following natural extension of the conjecture holds 
in general. 
\begin{conjecture}[Generalized Faudree-Lehel Conjecture\cite{FL}]\label{conj:main2}
There is a constant $C>0$ such that for all  graphs $G$ on $n$ vertices with minimum degree 
$\dd\geq 1$ and without isolated edges, 
\[s(G) \leq \frac{n}{\dd}+C.\] 
\end{conjecture}
It is this conjecture that ``energized the study of the irregularity strength'', as stated in~\cite{CL}, and settled foundations for entire discipline, as it provided an inspiration for
many related papers, concepts and intriguing questions, see e.g.~\cite{ Louigi30,Louigi2,AT,Amar,AKP,Baca,BarGrNiw,BMT-irregular,Bohman_Kravitz,CaoLu,CL,Dinitz,Faudree2,Faudree,Ferrara,FGKP,Gyarfas,KKP,KalKarPf_123,123KLT,L-survey,MP,Majerski_Przybylo,
Pasy,Przybylo_asym_optim,LocalIrreg_2,
Plinear1,Plinear2,1234Reg123,ThoWuZha,
WongZhu23Choos, AW}. 
The conjecture 
remains open after more than three decades since it was formulated. 
A significant step forward regarding it was achieved in 2002 by Frieze, Gould, Karo\'nski, and Pfender~\cite{FGKP}, who
used the probabilistic method to prove the first linear in $n/d$ bound $s(G) \leq 48(n/d)+1$ for $d \leq \sqrt{n}$, and a superlinear bound $s(G) \leq 240(\log n)(n/d)+1$ when $d > \integer{\sqrt{n}}$. Similar bounds for general graphs, with $d$ replaced by the minimum degree $\delta$, were also proved in the same paper. These in particular imply that 
$s(G) = O(n/\delta)$ if $G$ has maximum degree $\Delta \leq n^{1/2}$. The linear bounds in $n/d$ and $n/\delta$ 
were further extended to the cases when $d \geq 10^{4/3}n^{2/3}\log^{1/3}n$ and $\delta \geq 10n^{3/4}\log^{1/4}n$, respectively, by  Cuckler and Lazebnik~\cite{CL}. The first linear bounds in both $n/d$ and $n/\delta$ for all ranges of 
$n$, $d$ and $\delta$ were settled by Przyby{\l}o, who used a different idea to 
improve a key combinatorial lemma in~\cite{FGKP}, 
thus proving~\cite{Plinear1, Plinear2} that $s(G) \leq 16(n/d)+6$ and, resp., $s(G) \leq 112(n/\delta)+28$. 
Since then, 
considerable efforts were devoted to improve the multiplicative constant in front of $n/d$ and $n/\delta$.
In course of work over this and several related concepts a list of inventive and highly useful 
algorithms were developed in particular in~\cite{AKP, Kthesis, KKP, KalKarPf_123, MP}.
These assured important breakthroughs concerning $s(G)$ and other widely studied graph invariants. 
The best result among these is due to Kalkowski, Karo{\'n}ski and Pfender~\cite{KKP}, who proved that in general  $s(G) \leq 6 \ceiling{n/\delta}$ (what was later improved to $s(G) \leq (4+o(1))(n/\delta)+4$ for a narrower range of 
$\delta \geq \sqrt{n}\log n$ 
in~\cite{MP}).
It was only until recently when Przyby{\l}o~\cite{Pasy} proposed an algorithm 
which significantly improved the previous upper bounds for $d$-regular graphs, implying
that 
Conjecture~\ref{conj:main} holds asymptotically (in terms of $d$ and $n$)  for $d$ not in extreme values. 
\begin{theorem}[Przyby{\l}o~\cite{Pasy}]\label{thm:previous}
Given any fixed 
$\epsilon>0$, for every $d$-regular graph $G$ with $n$ vertices and $d\in [\ln^{1+\epsilon} n, n/\ln^{\epsilon}n]$, if $n$ is sufficiently large,
\[
s(G) \leq \frac{n}{d}\left(1+ \frac{1}{\ln^{\epsilon/19}n}\right).
\]
\end{theorem} 

In \cite{Pasy} Przyby{\l}o mentioned that ``a poly-logarithmic in $n$ lower bound on $d$ is unfortunately unavoidable" within his approach. In this paper, we first extend the range of $d$ to bypass the poly-logarithmic in $n$ lower bound (and the upper bound too), providing at the same time a stronger upper bound on $s(G)$ for all $1\leq d \leq n-1$.

In the case of general graphs with minimum degree $\delta$ (Conjecture \ref{conj:main2})  instead of regular graphs with degree $d$, obtaining a good bound on $s(G)$ is considerably harder. Existing methods used to yield a bound for regular graphs either stop working for general graphs, or would result in a much worse bound. 
Prior to our result, no asymptotically sharp bound on $s(G)$ has been shown for general graphs with minimum degree $\delta$. One exception is the family of random graphs $G(n,p)$ where $p$ is any fixed constant in $(0,1)$, where the first author showed $s(G) \leq \lceil n/\delta \rceil +2$ almost surely \cite{Prandom}.
In this paper, we are able to  show
 that the generalized Faudree-Lehel Conjecture (Conjecture~\ref{conj:main2}) also holds asymptotically.

\begin{theorem}\label{thm:main2}
For every $\epsilon\in (0,0.25)$, there are absolute constants $C_1, C_2$ such that for each graph $G$ with $n$ vertices and minimum degree $\dd>0$ which does not contain isolated edges, 
\[s(G) \leq \frac{n}{\delta}\left(1+ \frac{C_1}{\delta^{\epsilon}}\right)+C_2.\]
\end{theorem}

Note that Theorem~\ref{thm:main2} in particular implies (for $\varepsilon=0.2$) that $s(G)\leq (n/\delta)(1+(C_1/\delta^{0.2}))+C_2$ for some absolute constants $C_1,C_2$.
For any fixed $\varepsilon>0$ (in Theorem~\ref{thm:previous}) and $\delta\in [\ln^{1+\epsilon} n, n/\ln^{\epsilon}n]$, we however have: 
$(n/\delta)(1/\ln^{\varepsilon/19}n) \gg (n/\delta)(C_1/\ln^{0.2+0.2\varepsilon}n) \geq (n/\delta)(C_1/\delta^{0.2})$
and $(n/\delta)(1/\ln^{\varepsilon/19}n) \geq (\ln^\varepsilon n) (1/\ln^{\varepsilon/19}n) \gg C_2$,
and thus our bound is in particular a direct improvement over the bound in Theorem \ref{thm:previous} (which additionally regards only the case of regular graphs).
 

Moreover, as the second contribution of the paper, which seems even more or equally vital as the one above, 
we also 
confirm that the Faudree-Lehel Conjectures (Conjectures \ref{conj:main} and \ref{conj:main2})
hold, not only asymptotically, for 
all graphs with $\delta\geq n^\beta$ where $\beta$ is any fixed constant greater than $0.8$.

\begin{theorem}\label{thm:main1}
For every $\epsilon \in (0, 0.25)$, there is an absolute constant $C$ such that for each graph $G$ with $n$ vertices and minimum degree $\delta$ where $\delta^{1+\epsilon} \geq n$ (i.e., $\delta\geq n^{1/(1+\epsilon)}$),   \[s(G) \leq \frac{n}{\delta} + C.\]
\end{theorem}

\section{Tools and Notation}
We will use the following tools.
\begin{lemma}[Chernoff Bound c.f., e.g., 
\cite{ASbook}, Appendix A]\label{bound:chernoff}
Let $X_1, \dots, X_n$ be i.i.d.~random variables such that $\Pr(X_i=1) = p$ and $\Pr(X_i=0)=1-p$. Then for any $t \geq 0$,
\begin{align*}
\Pr\left(\left|\sum_{i=1}^n X_i - np\right| > t\right) \leq 2e^{-t^2/(3np)}, \ \ & \text{ if } 0 \leq t \leq np, \\
\Pr\left(\left|\sum_{i=1}^n X_i - np\right| > t\right) \leq 2e^{-t/3}, \ \ & \text{ if } t > np.
\end{align*} 
\end{lemma}
\begin{corollary}[Chernoff Bound c.f., e.g., 
\cite{ASbook}, Appendix A]\label{bound:chernoff2}
Let $X_1, \dots, X_n$ be i.i.d.~random variables such that $\Pr(X_i=1) = p$ and $\Pr(X_i=0)=1-p$. Then for any $t \geq 0$,
\begin{align*}
\Pr\left(\left|\sum_{i=1}^n X_i - np\right| > t\right) \leq 2e^{-t^2/3\max(np, t)}.
\end{align*} 
\end{corollary}


\begin{definition}[Negative Association \cite{DR}]
A set of random variables $X_1, \dots,X_n$ are {\it negatively associated} if for any two disjoint index sets $I, J \subset [n]$ and two  monotone increasing functions $f,g: \mathbb{R} \to \mathbb{R}$, $\ee[f(X_i, i\in I) g(X_j, j\in J)] \leq  \ee[f(X_i, i\in I)] \ee[ g(X_j, j\in J)].$
\end{definition}

\begin{lemma}[Zero-one Prinicple \cite{DR}]\label{lem:NAzero-one}
Let $X_1, \dots,X_n$ be zero-one random variables such that always $\sum_i X_i \leq 1$. Then $X_1, \dots, X_n$ are negatively associated.
\end{lemma}

\begin{lemma}[Closure Property \cite{DR}]\label{lem:NAcp}
Let $X_1, \dots,X_n$ be negatively associated and let $Y_1, \dots, Y_n$ be negatively associated. If $\{X_i\}_i$ are independent from $\{Y_i\}_i$, then $\{X_i\}_i \cup \{Y_i\}_i$ are negatively associated.
\end{lemma}
\begin{lemma}[Chernoff Bound \cite{DR}]\label{lem:NAChernoff}
Let $X_1,\dots, X_n$  be negatively associated random variables such that $\Pr(X_i=1)=p$ and $\Pr(X_i=0)=1-p$. Then the two
Chernoff Bounds in Lemma~\ref{bound:chernoff} hold.
\end{lemma}


\begin{lemma}[Simple Concentration Bound \cite{MolloyReed}]\label{lem:SCB}
Let $S$ be a random variable determined by $n$ independent trials $X_1,\ldots,X_n$, and satisfying:
changing the outcome of any one trial can affect $S$ by at most $c>0$, 
then for any $t>0$, 
\begin{align*}
\Pr\left(\left|S - \ee(S)\right| > t\right) \leq 2e^{-t^2/(2c^2n)}.
\end{align*} 
\end{lemma}

Given a graph $G$, 
a set $U \subset V(G)$ and a vertex $v \in V(G)$, we use $\deg_U(v)$ to denote the number of neighbors of $v$ in $U$. We use $G[U]$ to denote the 
subgraph induced by $U$ in $G$; 
and $E(U)$ to be the set of edges of $G[U]$.

Throughout the paper we assume that the graph $G$ has $n$ vertices and minimum degree $\delta$. 
A \emph{weighted degree} of a vertex $v$ will usually be abbreviated as a \emph{weight} of $v$.

\section{General Proof Idea, Links and Obstacles}
The basic intuition behind our construction is to partition $V(G)$ into a big set $B$ 
and a small set $S$, where $|S| = (n/\dd)\cdot o(\dd)$. We first adjust the edge weights so that almost all vertices in $B$ have distinct weights, and then we locally adjust weights of the rest of the vertices to distinguish the weights of 
all vertices in $G$. 

Our argument 
can be divided into three main steps.
Step A relies on a specific random construction which assures relatively sparse distribution of weights of the vertices in $B$, i.e. without too many vertex weights in any of the predefined 
intervals partitioning 
positive integers.  
Step B consists of modifications of the weights of edges across $B$ and $S$ to generate relatively small shifts of the vertex weights in $B$, resulting in pairwise distinct weights attributed to all but a small set of bad vertices in $B$. 
(We note here that $S$ must be large enough to provide sufficiently many edges across $B$ and $S$ for our purposes.)
In step C we modify weights of the edges in $S$ and {   a small portion of the edges outside $S$} to weight distinguish the vertices in $S$ mostly. For this aim we attribute these vertices special weights deliberately unused in step B (with residues at most $5$ modulo a carefully chosen and large enough integer $k$). To distinguish weights in $S$ we 
in particular
benefit from the fact that this set 
is small compared to $B$ and thus vertices in $S$ have on average large fraction of all their incident edges in $E(S,B)$ (usually much larger than edges in $S$).
This enables taking on vital preparatory measures prior to step C (within step A) ensuring sparse vertex weight distribution in $S$ and facilitating the mentioned final cleanup in this set. 
Throughout the construction we moreover single out several types of ``bad vertices'', which do not fulfill one of a number of specified conditions and cannot be weight distinguished according to major procedures. 
The set of all of these is however small enough to be handled with in a special manner in step C.
%
%


Our approach is motivated by the random construction idea from~\cite{Pasy}, 
which amounts to 
show that under certain conditions  there are no ``bad vertices" at all (in case of regular graphs).
Then an explicit weight assignment could be provided in the face of absence of such problematic vertices. 
Typically, if the minimum degree $\delta$ is $\Omega(\log n)$, 
then a union bound could be used to prove that with positive probability there are no ``bad vertices" resulting from the random construction. To bypass the $\log n$ factor, careful quantization and Lov\'asz Local Lemma turned out to be very helpful tools 
in the case of regular graphs -- in~\cite{WEI} we provide a significantly more simple approach, yielding similar results as the ones in this paper 
but for the setting
restricted to regular graphs exclusively. 
To prove the asymptotic bound for all $\dd$ in the case of general graphs, one of the real challenges is that the maximum degree and minimum degree could differ by any factor. 
This in particular forefends a direct application of the Lov\'asz Local Lemma.
In this paper we bypass all these difficulties. One of our main ideas is that although we cannot guarantee that there are no ``bad vertices" at all resulting from the random construction, yet the number of such bad vertices cannot be too large (in fact it is usually exponentially small). We therefore can accumulate those bad vertices and treat them at the end of the proof via careful and technical analysis.

\section{Proof of Main Results}

\subsection{Set-up}

We will focus on proving Theorem~\ref{thm:main2}. Only at the very end of the paper do we comment on how it directly implies Theorem~\ref{thm:main1}. Let us thus fix $\epsilon$, corresponding to Theorem~\ref{thm:main2}, followed by
an auxiliary small constant $\alpha$ such that
\begin{equation}
\epsilon\in (0,0.25),~~~~~~{   \alpha\in(0,\epsilon)}~~~~~~{\rm and}~~~~~~{   2\epsilon+\alpha<0.5},\label{EpsilonAlpha}
\end{equation}
and a graph $G$.
%
Let $X_v\sim U[0,1]$, $v \in V(G)$ be i.i.d.~ uniform random variables.
These 
are used to separate the vertices into $\delta$ bins. For $1 \leq i \leq \dd-1$, let the $i$-th \emph{bin}  be defined as
\[B_i = \{v : (i-1)/\dd \leq X_v < i/\dd\},\] 
and set the last bin as $B_\dd  = \{v : (\dd-1)/\dd \leq X_v \leq 1\}$. In expectation, every $B_i$ includes $n/\dd$ vertices. 
For each vertex $v$, we define: 
\[{\rm the~ random~ variable~} Z(v) \in [1, \dd] ~{\rm to~ be~ the~ bin~ number~} i~ {\rm such~ that~} v\in B_i.\] 

Let the small set $S$ be the union of bins $B_{i}$ with  $i > \dd- \sstar$ where
\[\sstar: =\ceiling{\dd^{1/2+{   \epsilon+\alpha}}}.\] 
Thus the expected number of vertices in $S$ is $n\sstar/\dd \ll_{\dd} n$. Denote $B := V(G) \setminus S$ to be the big set.

In order to take on certain preparatory measures prior to Step C (within which we will finally distinguish weights of the vertices in $S$),
we further partition $S$ into $k'$ similar sized subsets where 
\[{    k' := \min\left( \ceiling{\frac{n}{400\dd}}, \ceiling{\frac{\dd^{\epsilon}}{2000}}  \right)},\] such that each subset consists of $\ceiling{\sstar/k'}$ or $\integer{\sstar/k'}$ bins. For $i=1$, let $S_1$ be the first $\integer{\sstar/k'}$ bins in $S$, i.e., $S_1 = \{ \bigcup B_j: \dd-\sstar  < j \leq \dd {   -\sstar +\integer{\sstar/k'}}\}$. 
Next  sequentially define $S_2, \dots, S_{k'}$ so that $S_j$ is the union of the first $\integer{\sstar/k'}$ or $\ceiling{\sstar/k'}$ (depending on $\sstar \mod k'$) consecutive yet ungrouped bins in $S$. 
Furthermore set
\[k: = \ceiling{k'/1000}~~~~{\rm and}~~~~a': = \ceiling{{   7}n/(\dd k)} .\]
These two parameters will be used in Step A. 
%

We will show that when $\delta$ or $n/\delta$ is smaller than a constant $c$, then Theorem \ref{thm:main1} holds. Thus
throughout the computations in the paper, unless otherwise stated, we will assume there is an absolute constant $c$ such that $\delta \geq c$, $n/\delta \geq c$ and $c$ is large enough so that all explicit inequalities in the computations will hold.

\subsection{Step A.}
 
\begin{definition}\label{F1definition}
We define a weighting assignment $f_1: E(G) \to \mathbb{Z}$ in Step A as follows. 
\begin{enumerate}
\item For every  bin number $1 \leq i \leq \delta - \sstar$, for each vertex $v\in B_i \subset B$ and each of its neighbors $u$ in  $\{\bigcup B_j, {\delta-s^*-i+1< j \leq \delta-s^*}\}\subset B$, let $f_1(\{u,v\})$ be equal to $\ceiling{n/\dd}+a'$. 
\item For every integer $1 \leq j \leq k'$, for each vertex $v\in B$ and its neighbor $u \in S_j$, let $f_1(\{u,v\})$  be equal to $ \ceiling{ \ceiling{n/\dd}/{3k'}  }(j+k').$ 
\item Let $f_1$ equal $1$ on the rest of the edges in $B$, and let $f_1$ equal $0$ on the edges in $S$. 
\end{enumerate}
Let $\fv_1: V(G) \to \mathbb{Z}$ evaluate 
the weights of vertices under $f_1$. 
\end{definition}
\begin{definition}
Let $\sigma(v,i)$ be the random vertex weight $\fv_1(v)$ given $v \in B_i$. 
Let 
$I_0= (0, (\ceiling{n/\dd}+a'-1)),$
whereas for each integer $1 \leq h < {    2}n$, we define the interval \[I_h= [ h(\ceiling{n/\dd}+a'-1), (h+1)(\ceiling{n/\dd}+a'-1)).\] 
For each $1\leq i \leq \dd-\sstar$ and each vertex $v$, denote the expected weight of $v$ if $v\in B_i$ as 
\[ \sm(v,i) = \ee\left[\fv_1(v) \given v\in B_i\right]\] 
{and set $\sm(v,i)=0$ 
 for $\dd-\sstar < i\leq \delta$ and $v\in V(G)$}.
{   For $0\leq h < 2n$}, let $\mu_h$ be the expected number of vertices in $G$ such that $\sm(v, Z(v)) \in I_h$ {   (note all these vertices must belong to $B$)}. 
Given integers $0 \leq h_1 \leq h_2$, let 
\[\mu_{[h_1, h_2)} = \sum_{h=h_1}^{h_2-1}\mu_h.\]
\end{definition}

Note that since $0 \notin I_0$, then $\sigma_\mu(v, i) \notin I_h$ for all $h \geq 0$ if $i > \delta - \sstar$.

By Definition~\ref{F1definition}, 
\begin{equation}
\| \fv_1\|_\infty\leq \max_v \deg(v) (\ceiling{n/\dd}+a') \leq (n-1)(\ceiling{n/\dd}+a')  < 2n(\ceiling{n/\dd}+a'-1),
\label{eq:summuh}
\end{equation}
and thus $\mu_{[0, 2n)}  =\ee(|B|)$.

\begin{claim}\label{claim:smvi} 
Let $v$ be a fixed vertex and $i$ an integer in $[1, \dd-\sstar]$. Given $v \in B_i$, with probability at most $2e^{-\dd^{2{   \alpha}}/2}$, $\left| \sigma(v,i)- \sm(v,i)\right| > \deg(v)^{1/2+{   \alpha}} \left( \ceiling{\frac{n}{\dd}}+a'-1\right)$. In addition, for each vertex $v$ and a fixed integer $h \in [0, 2n)$, there is at most one bin $i$ such that ${   \sm}(v,i)\in I_h$. This implies 
$\mu_h \leq  n/\dd$.
\end{claim}
\begin{proof}
Fix any  $i\in [1, \dd-\sstar]$. 
Given $v\in B_i$, the expected value of $\sigma(v,i)$ equals $\sm(v,i)$ and is determined by the values of $\deg(v)$ independent variables $X_v$, $v\in N(v)$. As by definition of $f_1$, changing the outcome of any one of these variables can affect $\sigma(v,i)$ by at most $ \left(\ceiling{\frac{n}{\dd}}+a'-1\right)$, the Simple Concentration Bound (Lemma~\ref{lem:SCB}) implies that the probability that $\left| \sigma(v,i)- \sm(v,i)\right| > \deg(v)^{1/2+{   \alpha}} \left( \ceiling{\frac{n}{\dd}}+a'-1\right)$ is at most
\begin{align*}
2e^{-\deg(v)^{1+2{   \alpha}} \left( \ceiling{\frac{n}{\dd}}+a'-1\right)^2/(2\left( \ceiling{\frac{n}{\dd}}+a'-1\right)^2\deg(v))}
\leq 2e^{-\delta^{2{   \alpha}}/2},
\end{align*} 
as desired.

To prove our second claim, we need to approximate $\sm(v,i)$. This is equal to $w_\ee(v,i, B)+ w_\ee(v,i, S)$ where $w_\ee(v,i, B)$ and  $w_\ee(v,i, S)$ denote  the expected weights of $v$ coming from its neighbors in $B$ and $S$, respectively, given that $v\in B_i$. Let us denote for simplicity: $B^i = \{\cup B_j, \delta-\sstar  - i + 1 < j \leq \delta- \sstar\}$.
The probability that a fixed neighbor of $v$ lies in  $B^i$ 
equals exactly $(i-1)/\dd$. Thus the expected number of neighbors of $v$ lying  in $B^i$ 
equals $\deg(v) \frac{i-1}{\dd}$. Similarly, the expected number of neighbors of $v$ lying in $B \setminus B^i$ 
equals $\deg(v) \frac{\dd - \sstar - i+1}{\dd}$. Therefore, 
\begin{align}
w_\ee(v,i, B)= & 
(\lceil n/\dd \rceil + a') \ee[\deg_{B^i}(v)] +1 \cdot \ee[\deg_{B \setminus B^i}(v)]  \nonumber \\
 = & 
\deg(v) \frac{i-1}{\dd}\left( \ceiling{\frac{n}{\dd}}+a'\right) + 1 \cdot \deg(v) \frac{\dd - \sstar - i+1}{\dd} \nonumber \\
= &
\frac{  \deg(v) }{\dd} \left( (i-1)\left(\ceiling{\frac{n}{\dd}}+a' -1 \right) +\dd -\sstar \right).   \label{eq:weightvB}
\end{align}
On the other hand, by condition 2 in the definition of $f_1$, the expected value $w_\ee(v,i, S)$ does not depend on $i$, and thus is equal to some number $w_\ee(v,S)$, whose precise value is irrelevant within our proof.
Combining this fact with~(\ref{eq:weightvB}), we obtain that
\begin{align}
\sm(v,i) = 
\  w_\ee(v,i, B)+ w_\ee(v,i,S) 
= 
\ \frac{\deg(v)}{\dd} \left(   (i-1)\left( \ceiling{\frac{n}{\dd}}+a'-1\right) 
 +\dd-\sstar\right) + w_\ee(v,S). \nonumber 
\end{align}
Therefore,
 $\sm(v,i)\in I_h$ implies that 
\begin{align*}
h\left(\ceiling{\frac{n}{\dd}}+a'-1\right)\leq 
\frac{\deg(v)}{\dd} \left(   (i-1)\left( \ceiling{\frac{n}{\dd}}+a'-1\right) 
 +\dd-\sstar\right)  + w_\ee(v,S) 
  <(h+1)\left(\ceiling{\frac{n}{\dd}}+a'-1\right).
\end{align*}
Rearranging the inequalities, we have:
\[
\frac{\dd h}{\deg(v)} - \frac{\dd -\sstar + \frac{\delta}{\deg(v)}w_\ee(v,S)}
{\ceiling{\frac{n}{\dd}}+a'-1} \leq i - 1   < \frac{\dd (h+1)}{\deg(v)} 
- \frac{\dd -\sstar + \frac{\delta}{\deg(v)}w_\ee(v,S)}
{\ceiling{\frac{n}{\dd}}+a'-1}.
\]
Thus in order to make $\sm(v,i) \in I_h$, $i$ has to lie in an interval (with one end open) of length $\dd/\deg(v) \leq 1$. 
Therefore, given $h$ and $v$, there is at most one integer $i \leq \dd-\sstar$ so that $\sm(v,i) \in I_h$. The probability that $v$ lies in the corresponding bin (if exists) equals $1/\dd$. Thus 
$\mu_h \leq \sum_{v \in V(G)} (1/\dd) = n/\dd.
$
\end{proof}

\begin{definition}
Given an integer $0\leq h <  2n$, let 
$V_h$ be the set of vertices such that 
$\sm(v, Z(v)) \in I_h$ and 
$\left| \sigma(v,Z(v))- \sm(v,Z(v))\right| \leq \deg(v)^{1/2+{   \alpha}} \left( \ceiling{\frac{n}{\dd}}+a'-1\right)$. 

Let $\yb{h}$ be the set of bad vertices such that 
$\sm(v, Z(v)) \in I_h$ and 
$\left| \sigma(v,Z(v))- \sm(v,Z(v))\right|$ $>$ $ \deg(v)^{1/2+{   \alpha}}$ $\left( \ceiling{\frac{n}{\dd}}+a'-1\right)$. 

Given two integers $0\leq h_1 \leq  h_2 \leq 2n$, let 
\[V_{[h_1, h_2) }= \bigcup_{h_1 \leq h< h_2}V_h~~~~~~~~{\rm and}~~~~~~~~\yb{[h_1, h_2)}= \bigcup_{h_1 \leq h< h_2}\yb{h}.\]
\end{definition}

The sets $\yb{h}$ consist of the first kind of bad vertices in the proof. 
Clearly $V_h\cup \yb{h}$ is the set of all vertices $v$ such that $\sm(v,Z(v))\in I_h$, 
and thus, by~(\ref{eq:summuh}), $V_{[0,2n)} \cup \yb{[0, 2n)}=B$.

\begin{lemma}\label{lem:yb}
Given two integers $0\leq h_1 <h_2 \leq 2n$ such that $\mu_{[h_1, h_2)} > 1$,  the following inequalities hold simultaneously with probability at least 
$1- e^{-\dd^{2\alpha}/4} -2\exp\left(
-\frac{1}{{3}}
\sqrt{\frac{n}{\dd^{1-2{\alpha}}}}
\right)$:
\begin{align}
& { \left|V_{[h_1, h_2)}\right| \leq   \mu_{[h_1, h_2)}   + \sqrt{ \frac{n\mu_{[h_1, h_2)}}{\dd^{1-2{\alpha}}}}; }  \label{eq:sizeVh} \ \ \text{and} \\
& \left|  \yb{[h_1, h_2)}  \right| {< 2}\mu_{[h_1, h_2)}  {\exp({-\dd^{2{\alpha}}/4})}. \label{eq:sizeyb}
\end{align}
In addition, if $\dd > \sqrt{n}$, with probability at least $1-1/{n}$, 
\begin{align}
& \yb{[0, 2n)}  = \emptyset. \label{eq:sizeyb2}
\end{align}
\end{lemma}
\begin{proof}
For each vertex $v$ and integers $i \in [1, \dd-\sstar]$, $h\in[0,{2}n)$, let a random variable
$Y_{v,i,h}$ be the indicator function that $v\in B_i$ and $\sm(v, i)\in I_h$; 
define
$Z_{v,i,h}$ to be equal to $1$ if $v\in B_i$,  $\sm(v, i)\in I_h$ and $\left| \sigma(v,i)- \sm(v,i)\right| >  \deg(v)^{1/2+{   \alpha}} \left( \ceiling{\frac{n}{\dd}}+a'-1\right)$, and $Z_{v,i,h}=0$ otherwise.
Thus 
\begin{align}
& |V_h| = \sum_{v\in V(G), 1\leq i \leq \dd-\sstar} (Y_{v,i,h}-Z_{v,i,h}) ~~~
{\leq \sum_{v\in V(G), 1\leq i \leq \dd-\sstar} Y_{v,i,h},}  \label{eq:VhX}\\
&  |\yb{h}|= \sum_{v\in V(G), 1\leq i \leq \dd-\sstar}Z_{v,i,h}, \label{eq:ybsumZ} \\
& \ee[\sum_{h_1 \leq h < h_2}\sum_{v\in V(G)}\sum_{1\leq i\leq \dd- \sstar} Y_{v,i,h}]=
\sum_{h_1 \leq h < h_2}\sum_{v\in V(G)}\sum_{1\leq i\leq \dd- \sstar} \Pr(v \in B_i, \sm(v,i)\in I_h)
=
\ \mu_{[h_1, h_2)}. \label{eq:Ymu}
\end{align}

By Claim \ref{claim:smvi},
\begin{align}
\Pr(Z_{v,i,h}=1) = & \Pr\left(
\left| \sigma(v,i)- \sm(v,i)\right| >  \deg(v)^{1/2+{\alpha}} \left( \ceiling{\frac{n}{\dd}}+a'-1\right)   \given[\Big]{ v \in B_i} \right) \cdot \Pr(v\in B_i, \sm(v,i)\in I_h) \nonumber  \\
\leq & {2e^{-\dd^{2{\alpha}}/2}} \Pr(v\in B_i, \sm(v,i)\in I_h).\label{eq:Zprosmall}
\end{align}
Thus together with (\ref{eq:ybsumZ}), (\ref{eq:Zprosmall}) and (\ref{eq:Ymu}),
\begin{align}
\ee(|\yb{[h_1, h_2)}|) =  &\sum_{h_1 \leq h < h_2} \sum_{v\in V(G)}\sum_{1\leq i \leq \dd-\sstar}\ee(Z_{v,i,h}) =  \sum_{h_1 \leq h < h_2} \sum_{v\in V(G)}\sum_{1\leq i \leq \dd-\sstar}\Pr(Z_{v,i,h}=1) \nonumber \\
\leq & {2e^{-\dd^{2{\alpha}}/2}} \sum_{h_1 \leq h < h_2}\sum_{v\in V(G)} \sum_{1\leq i \leq \dd-\sstar}\Pr(v\in B_i, \sm(v,i)\in I_h)
=
 {2e^{-\dd^{2{\alpha}}/2}}  \mu_{[h_1, h_2)}. \nonumber
\end{align}

Therefore by Markov's Inequality,
\begin{equation}
\Pr\left(|\yb{[h_1, h_2)}| \geq  {2e^{-\dd^{2{\alpha}}/4}}  \mu_{[h_1, h_2)}\right) \leq \frac{\ee(|\yb{[h_1, h_2)}|)}{ {2e^{-\dd^{2{\alpha}}/4}}  \mu_{[h_1, h_2)}} \leq \frac{ {2e^{-\dd^{2{\alpha}}/2}} \mu_{[h_1, h_2)}}{ {2e^{-\dd^{2{\alpha}}/4}}  \mu_{[h_1, h_2)}}=  {e^{-\dd^{2{\alpha}}/4}} .  \label{eq:ybclose}
\end{equation}

For the sake of~(\ref{eq:sizeyb}) we will now bound
the probability that $\sum_{h_1 \leq h < h_2}\sum_v\sum_{1\leq i\leq \dd- \sstar} Y_{v,i,h}$ is 
far from its expectation in (\ref{eq:Ymu}).
Note that for a fixed $v$, although there is no bin $B_i$ that vertex $v$ must lie in, given any random experiment
(evaluating $X_v$) 
there is exactly one bin $B_i$ that $v$ belongs to, while 
$\deg(v)$ together with the bin $B_i$ determine the unique at most one value of $h$ for which $\sm(v,i)\in I_h$. 
Thus 
the indicator random variables satisfy $\sum_{i, h} Y_{v,i,h} \leq 1$, and hence, by Lemma \ref{lem:NAzero-one}, the random variables $\{Y_{v,i,h}\}_{i,h}$ are negatively associated. Furthermore, since the variables in $\{Y_{v,i,h}\}_{i,h}$ are independent from such variables corresponding to other vertices, by Lemma \ref{lem:NAcp}, the 
random variables in $\{Y_{v,i,h}\}_{v, h_1 \leq h < h_2, 1\leq i\leq \dd- \sstar}$ are negatively associated. Thus, leaving out the random variables $Y_{v,i,h}$ which are constantly zero, the rest are identically distributed and negatively associated. Hence we may use the Chernoff Bound (Lemma~\ref{lem:NAChernoff})  and the shorthand $Y =\sum_{h_1 \leq h < h_2}\sum_v\sum_{1\leq i\leq \dd- \sstar} Y_{v,i,h}$, where by (\ref{eq:Ymu}), $\ee[Y] = \mu_{[h_1, h_2)}$: 
\begin{align}
& \Pr\left(\left|Y - \mu_{[h_1, h_2)}\right|  > 
\sqrt{\frac{n\mu_{[h_1, h_2)}}{{\dd^{1-2{\alpha}}}}} \right) 
\leq  2\exp\left(-\frac{ n\mu_{[h_1, h_2)}/{\dd^{1-2{\alpha}}}} {3 \max\left(\sqrt{n\mu_{[h_1, h_2)} /{\dd^{1-2{\alpha}}}}, \mu_{[h_1, h_2)} \right)}\right) \nonumber\\ 
\leq & 2\exp\left(
-\min \left(
\frac{\sqrt{n\mu_{[h_1, h_2)} /{\dd^{1-2{\alpha}}}}}{3}, \frac{n}{{3}\dd^{1-2{\alpha}}}
\right)
\right) \leq 2\exp\left(
-
\frac{\sqrt{n/{\dd^{1-2{\alpha}}}}}{3}
\right).
\label{eq:Yclose}
\end{align}

Inequalities~(\ref{eq:ybclose}) and (\ref{eq:Yclose}) thus imply 
~(\ref{eq:sizeVh}) and~(\ref{eq:sizeyb}) (cf.~(\ref{eq:VhX})).
%
%
%
Moreover, by (\ref{eq:Zprosmall}), 
\begin{align}
\Pr(v \in \yb{[0,2n)}) = & \sum_{1\leq i\leq \dd-\sstar, h\in[0,2n)} \Pr(Z_{v,i,h}=1)  
\leq {2e^{-\dd^{2{\alpha}}/2}} 
\sum_{1\leq i\leq \dd-\sstar} \sum_{h\in[0,2n)}\Pr(v\in B_i, \sm(v,i)\in I_h) \\
\leq & {2e^{-\dd^{2{\alpha}}/2}} 
\sum_{1\leq i\leq \dd-\sstar} \Pr(v\in B_i) <{2e^{-\dd^{2{\alpha}}/2}.} 
\end{align}
Thus by a union bound over $v$, when $\dd > \sqrt{n}$, $\Pr\left(\yb{[0,2n)} = \emptyset\right) \geq 1 - n 
{2e^{-\dd^{2{\alpha}}/2}} 
> 1- 1/{n}$. 
\end{proof}

Since the degree distribution 
can vary in $G$, it will be useful to group $\mu_h$'s of smaller sizes. 
As by Claim~\ref{claim:smvi}, $\mu_h \leq n/\dd$ for each integer $h$, we may define the following benchmarks to that end.
\begin{definition}
We can sequentially define the {\it benchmarks} $\bench{1},\bench{2}, \dots$ such that $\bench{1}=0$ and 
\begin{equation}
n/(2\dd) < \mu_{[\bench{i},\bench{i+1})} \leq {2n/\dd}. 
\label{eq:bench}
\end{equation} 
\end{definition}
\begin{claim}\label{claim:bench}
The number of benchmarks is 
at most ${2\dd}$.
\end{claim}
\begin{proof}
Since $\sum_{0\leq h<2n} \mu_h {\leq} n$, 
the claim holds by (\ref{eq:bench}) (and~(\ref{eq:summuh})). 
\end{proof}


By (\ref{eq:bench}) and Claim~\ref{claim:bench} we may apply Lemma \ref{lem:yb} to at most ${(2\dd)^2 = 4\dd^2}$ distinct benchmark pairs $(\bench{i}, \bench{j})$ in order to conclude that the following union bound holds.
\begin{corollary}\label{cor:yb}
With probability at least $1-{4}\dd^2 {e^{-\dd^{2{\alpha}}/4}} 
- {8}\dd^2\exp\left(
-\frac{1}{{3}}
\sqrt{\frac{n}{\dd^{1-2{\alpha}}}}
\right)$,
for any two benchmarks $\bench{i} {   <} \bench{j}$,  the following two inequalities hold simultaneously:
\begin{align}
& { \left|V_{[\bench{i}, \bench{j})}\right|  \leq  \mu_{[\bench{i}, \bench{j})}  +\sqrt{\frac{n \mu_{[\bench{i}, \bench{j})}}{\dd^{1-2{\alpha}}}}};  \label{eq:sizeVh2} \ \ \text{and} \\
& \left|  \yb{[\bench{i}, \bench{j})}  \right|  
{ < 2}\mu_{[\bench{i}, \bench{j})}  {\exp({-\dd^{2{\alpha}}/4})}.
\label{eq:sizeyb2}
\end{align}
In addition, if $\dd >\sqrt{n}$, with probability at least $1- 1/{n}$, $\yb{[0,2n)} = \emptyset$.
\end{corollary}

\subsection{Step B.} 
\subsubsection{Preparations}
Prior to performing Step B, we will expose that 
with high probability, all except a small fraction of vertices have relatively large degrees to $S$. 
We thus define below new types of bad vertices.
\begin{definition}\label{def:ys}
Let $\ys$ 
be the set of vertices $v\in V(G)$ with {less than} %
 $\sstar \deg(v)/(2\dd)$ neighbors in $S$.  
Let $\ysn \subset B$ be an arbitrary set of vertices such that each vertex $v \in \ys$ has at least 
$\sstar/2$ neighbors in $\ysn$ and $|\ysn| \leq {\ceiling{\sstar/2}} |\ys|$. 
\end{definition}
To see that such set $\ysn$ exists, note that if a vertex has 
at most 
$\sstar \deg(v)/(2\dd)$ neighbors in $S$, then it has at least $\deg(v) - \sstar \deg(v)/(2\dd) > \sstar/2$ neighbors in $B$. 
For each $v \in \ys$, we may thus choose arbitrary {$\ceiling{\sstar/2}$} neighbors of $v$ {in $B$} and add these to $\ysn$. 
Consequently, $|\ysn| \leq {\ceiling{\sstar/2}} |\ys|$.
Note there might be vertices of $\ys$ that are in $S$. 

\begin{lemma}\label{lem:ysysn}
With probability at least $1- \exp(-\sstar/24)- 2\exp\left(   - n/({4}\delta^{0.5})\right) $, the following statements hold:
\begin{align}
&\left|\left|S\right| - \sstar n / \dd\right| \leq n/\delta^{0.5-\epsilon}, \label{IneqS}\\
& |\ys| { <} 2ne^{-\sstar/24}, \ \ \text{and thus} \ \    |\ysn| {<} {2} n \sstar e^{-\sstar/24}. \label{IneqZS}
\end{align}
\end{lemma}
In addition, if $\dd > \sqrt{n}$, with probability at least $1-1/{n}$, ${\ys} = \emptyset$ and thus $\ysn = \emptyset$. 
\begin{proof}
For each vertex $v$, let $Z_v= 1$ if $v$ has {less than} 
$\sstar \deg(v)/(2\dd)$ neighbors in the random set $S$, and $Z_v=0$ otherwise. Thus $
|\ys |= \sum_v Z_v
$.

Fix $v \in {V(G)}$; each of its $\deg(v)$ neighbors has independently 
probability $\sstar/\dd$ to be in $S$. The expected number of neighbors of $v$ in $S$ is thus $\deg(v) \sstar/\dd$. By the Chernoff Bound (Lemma \ref{bound:chernoff}), 
\begin{align}
& \Pr(Z_v = 1)= \Pr\left(\deg_S(v) { <}  { 0.5}\ee[\deg_S(v)] \right)  
\leq 
\Pr\left(  \left|  \deg_S(v) - \ee[\deg_S(v)] \right| { >}  0.5 \ee[\deg_S(v)] \right) \nonumber \\
\leq & 2\exp\left( - \ee[\deg_S(v)]/12 \right) =  2\exp\left( - \deg(v)\sstar/(12\dd) \right) \leq 2\exp\left( - \sstar/12 \right). \label{eq:zv1}
\end{align}
Since
$
|\ys |= \sum_v Z_v
$, by (\ref{eq:zv1}) we have: \begin{equation}
\ee[|\ys|] =  \sum_v \Pr(Z_v=1) \leq 2n\exp\left( - \sstar/12 \right). \nonumber 
\end{equation} 
Therefore, by Markov's Inequality, 
\begin{equation}
\Pr\left(|\ys| \geq 2n\exp\left( - \sstar/24 \right) \right)
 \leq
  \ee\left[|\ys|\right]/ \left( 2n\exp\left( - \sstar/24 \right) \right) \leq \exp(-\sstar/24). \label{PrZS}
\end{equation}
We are left to bound the probability that $\left|\left|S\right| - \sstar n / \dd\right| \leq n/\delta^{0.5-\epsilon}$. Each vertex  independently has probability $\sstar/\dd$ to be in $S$, and hence $\ee[|S|]=\sstar n/\dd \geq n/\delta^{0.5-\epsilon}$. The Chernoff Bound thus implies that by~(\ref{EpsilonAlpha}), 
\begin{eqnarray}
\Pr\left( \left|\left|S\right| - \sstar n / \dd\right| > n/\delta^{0.5-\epsilon}\right) 
&\leq& 2\exp\left({-(n/\delta^{0.5-\epsilon})^2/(3\sstar n/\dd)}\right) \nonumber\\
&{ \leq } &
{ 2\exp\left(   - n/({4}\delta^{0.5-\epsilon+\alpha}) 
\right)
 \leq } 2\exp\left(   - n/({4}\delta^{0.5}) 
\right). \label{PrS}
\end{eqnarray}
Thus by~(\ref{PrZS}) and~(\ref{PrS}), with probability at least $1- 2\exp\left(   - n/({4}\delta^{0.5})
\right) - \exp(-\sstar/24)$ the two desired statements~(\ref{IneqS}) and~(\ref{IneqZS}) hold. 

In addition, when $\dd > \sqrt{n}$, applying a union bound 
to (\ref{eq:zv1}) yields: 
$\Pr({\ys} = \emptyset) \geq$ 
$1-n \cdot 2\exp(-\sstar/12)> 1- 1/n$.
\end{proof}

\subsubsection{Weighting Step B.}
In this step we define a weight assignment 
$f_2: E(G) \to \mathbb{Z}$ that is only supported on edges across $B$ and $S$ and $\| f_2\|_\infty = o(n/\dd)$, which modifies initial weights appointed by $f_1$. The goal is that at the end of Step B, the weights $(\fv_1 + \fv_2)(v)$ are distinct for vertices $v$ in $B$ {  (at least for these which are not bad)}, and the vertex weights in $B$ are not equal to $0, 1, 2,3,4,5$ modulo $k$. Recall $k = \ceiling{k'/{  1000}}$. 
We may assume $k > 50$ (for sufficiently large $\dd$ and $n/\dd$). 

\paragraph{Step 1.} 
We first bound modifications necessary to set most vertex weights in $B$  at values expected after Step A.
{  We admit a small error though, as the expected values do not have to be integers.}
 
\begin{claim}\label{claim:stepB1}
We may construct $f_2: E(G) \to \mathbb{Z}$  supported on edges across $B \setminus \left( \yb{[0,{2}n)} \cup \ys  \right)$ and $S$ so that 
\begin{equation}
\|f_2\|_\infty \leq {  2}\left( \ceiling{\frac{n}{\dd}}+a'\right)\dd^{1/2+{   \alpha}} /\sstar +1~~~~~~
{\rm and} ~~~~~~{  \left|(\fv_1+\fv_2)(v) - \sm(v,Z(v))\right|\leq 1} \label{FirstF2Cond}
\end{equation}
for each $v \in B \setminus \left( \yb{[0,{2}n)} \cup \ys  \right)$.
\end{claim}
\begin{proof}
For each vertex $v \in B\setminus \yb{[0,{  2}n)}$, by the definition of $\yb{h}$ (and~(\ref{eq:summuh})), 
\begin{equation}\label{SigmaSigmaMi2} 
\left| \sigma(v,Z(v))- \sm(v,Z(v))\right| \leq \deg(v)^{1/2+{\alpha}} \left( \ceiling{\frac{n}{\dd}}+a'-1\right),
\end{equation}
i.e. the weight of $v$ needs to be modified by at most $\deg(v)^{1/2+{\alpha}} \left( \ceiling{\frac{n}{\dd}}+a'-1\right)$. 
If at the same time $v \notin \ys$, then it has at least $\deg(v)\sstar/(2\dd)$ neighbors in $S$. Therefore, by~(\ref{SigmaSigmaMi2}), in order to satisfy the second condition in~(\ref{FirstF2Cond}),  
it is sufficient to modify 
each edge between {$v\in B \setminus \left( \yb{[0,{2}n)} \cup \ys  \right)$} 
and its neighbors in $S$ 
by at most
\begin{align}
& \ceiling{\deg(v)^{1/2+{\alpha}} \left( \ceiling{\frac{n}{\dd}}+a'-1\right) / \deg_S(v)} < {\deg(v)^{1/2+{\alpha}} \left( \ceiling{\frac{n}{\dd}}+a'\right) / (\deg(v)\sstar/(2\dd))} +1  \nonumber \\
{=} & {2}\left( \ceiling{\frac{n}{\dd}}+a'\right)\dd /(\sstar \deg(v)^{1/2-{\alpha}}   ) + 1 \leq {2}\left( \ceiling{\frac{n}{\dd}}+a'\right)\dd^{1/2+{\alpha}} /\sstar +1. 
\end{align}
\end{proof}

\paragraph{Step 2.} 
Now we wish to 
modify $f_2$ so that all the vertices in $B \setminus (\yb{[0,2n)}\cup \ys)$ have distinct vertex weights under $\fv_1 + \fv_2$ which moreover are not equal 
to $0,1,\dots, 5$ modulo $k$. 

Recall $v\in V_h$ means $\sm(v, Z(v))\in I_h \subset \left[h\left( \ceiling{\frac{n}{\dd}}+a'-1\right), (h+1)\left( \ceiling{\frac{n}{\dd}}+a'-1\right)\right)$. 


\paragraph{Step 2-1.}
We first modify $f_2$ so that for $0 \leq h < { 2}n$, all the vertices in  every given $V_h \setminus (\yb{h}\cup \ys)$ have distinct weights not equal to $0,1,\dots, 5$ modulo $k$, fitting as many as possible of these weights in $I_h$ (note we will admit weights from different $V_h$'s to overlap for now).


Let $\zz' \subset \mathbb{Z}$ be the set of integers which are not equal to $0$ to $5$ modulo $k$. Elements in $\zz'$ inherit their natural ordering from $\zz$. Two integers are said to be consecutive in $\zz'$ if they are consecutive in their ordering in $\zz'$. Intervals in $\zz'$ are thus consecutive integers in $\zz'$ with respect to the ordering in $\zz'$. 
For each interval $I$ of consecutive $|I|$ integers in $\zz$, it is easy to see  that:
\begin{equation}
|I \cap \zz'| \geq (|I|-{ 6})(k-6)/k. \label{eq:countnotk}
\end{equation}
By (\ref{eq:countnotk}), 
the size of an 
interval $I \subset \zz$ does not change much after restricting it to $\zz'$, i.e.,
\begin{equation}\label{eq:countnotk2}
|I| \leq |I \cap \zz'|k/(k-6)+{ 6}.
\end{equation}
\begin{claim}\label{claim:stepB21}
For each integer $0 \leq h < { 2}n$, it is sufficient to modify the weight of every  $v\in V_h \setminus (\yb{h}\cup \ys)$ by at most  $\max(|V_h|, |I_h|)k/(k-6)+ { 12}$ in order to attribute distinct weights to all vertices in $V_h \setminus (\yb{h}\cup \ys)$ and guarantee that these form consecutive integers in $\zz'$ with the smallest vertex weight 
{ equal to $\min (I_h\cap \zz')$, that is} 
at most $h(\ceiling{n/\dd}+a'-1)+6$.
\end{claim}
\begin{proof}
For any given $h$, we analyze vertices in $V_h \setminus (\yb{h}\cup \ys)$ one after another.
%
By 
Claim \ref{claim:stepB1}, each such vertex $v$ 
has the current 
weight { at most one away from} $\sm(v,Z(v)) \in I_h$. 
Changing it 
by at most ${ 7}$ we may thus shift this weight inwards $I_h\cap \zz'$.
Next, in order to reach the least yet unoccupied position in $\zz'$ which is not smaller than  $\min (I_h\cap \zz')$ this weight needs to be further shifted by 
at most 
$\max(|V_h|, |I_h|){ -1}$ consecutive integers in $\zz'$.
Thus by   (\ref{eq:countnotk2}), the vertex weight of $v$ after Step 1 needs  to be changed in total by at most ${ 7  + ((} \max(|V_h|, |I_h|)k/(k-6)+ { 6)-1)}$. 
\end{proof}

\paragraph{Step 2-2.} In this step we discuss further modifications of $f_2$ resulting in pairwise distinct weights from $\zz'$ associated to all vertices in  $V_{[0, 2n)} \setminus (\yb{[0,2n)}\cup \ys)$. 
%
We will need the following combinatorial lemma concerning shifts of intervals sufficient to make them pairwise disjoint. 
\begin{lemma}\label{lem:integershift}
Let ${ p}$ be a positive integer. For any ${ p}$ intervals ${I_i}=[a_i, b_i) \subset \zz$ (i.e., $a_i,b_i\in\zz$), $1\leq i \leq { p}$ where $a_1 \leq a_2\leq \dots \leq a_{{ p}}$, there exist 
${ p}$ disjoint intervals $I_i'\subset \zz$  with $I_i' = [a_i', b_i')$ such that 
$b_i'-a_i' = b_i - a_i$ for 
$1 \leq i \leq { p}$, 
$a_1=a_1' \leq  a_2' \leq \dots \leq a_{{ p}}'$ and 
\begin{equation}
\max_{1\leq i \leq { p}} |a_i'-a_i|  \leq \max_{1\leq l_1 < l_2 \leq { p}-1} \left(  \sum_{i=l_1}^{l_2}  (b_i-a_i) - (a_{i+1}-a_{i})   \right). \label{ShiftsIneq}
\end{equation}
\end{lemma}
\begin{proof}
As for any $i$, $|I_i'|=b_i' - a_i' = b_i-a_i = |I_i|$, an interval  $I_i'$ may be regarded as the translation of 
$I_i$ by $a_i'  - a_i$.
We will prove the lemma by induction on ${ p}$. For ${ p}=1$ the claim trivially holds.

Suppose ${ p} > 1$.  First consider the case when $a_2 \geq b_1$, or equivalently $I_1$ is disjoint from $I_2$. By the inductive hypothesis applied to the intervals $I_2, \dots, I_{{ p}}$, we could shift them to disjoint intervals $I_2', \dots, I_{{ p}}'$ where $I_2$ is the same as $I_2'$. 
With $I_1$ having no need to shift,
each of the remaining intervals has thus been 
shifted by at most $\max_{2\leq l_1 < l_2 \leq { p}-1} \left(  \sum_{i=l_1}^{l_2}  (b_i-a_i) - (a_{i+1}-a_{i})   \right)$.
The shifted intervals $I_2', \dots, I_{{ p}}'$ are moreover still disjoint from $I_1$ since for $i\geq 2$, $a_i' \geq a_2' = a_2 \geq b_1$, 
hence the lemma holds.

We may thus assume that 
$a_2 < b_1$, i.e. $I_1$ and $I_2$ overlap. Since $I_1'$ needs to remain the same as $I_1$, the interval $I_2$ must shift to the right by (at least) $b_1 - a_2 = (b_1-a_1) - (a_2-a_1)$, and become $I_2' =[a_2', b_2')= [b_1, b_1+ |I_2|)=[b_1, b_1+ (b_2-a_2))$. If $I_2'$ is to the left of $I_3$ and $I_2'$ is disjoint from $I_3$  (i.e., $a_3 \geq b_2'$),  by the inductive hypothesis applied to $I_3, \dots, I_{{ p}}$, we 
obtain a desired $\{I_i'\}_{1\leq i\leq { p}}$ with the maximum shift at most $\max\left(b_1-a_1- (a_2-a_1),  \max_{3\leq l_1 < l_2 \leq { p}-1} \left(  \sum_{i=l_1}^{l_2}  (b_i-a_i) - (a_{i+1}-a_{i})   \right) \right)$, which is bounded above by the right-hand side in~(\ref{ShiftsIneq}). 

If on the other hand $I_2'$ overlaps with $I_3$ or $I_2'$ is to the right of $I_3$ (i.e., $a_3 < b_2'$), then $I_3$ has to be shifted, e.g. to $I_3' =[a_3', b_3')= [b_2', b_2'+ (b_3-a_3))$. We analogously continue this process,  obtaining $I_j' = [a_j', b_j')$ sequentially
for $j=3,\ldots,l$ where ${ 3} \leq l \leq { p}$ is the last index such that $a_l < b_{l-1}'$. 
That is, for each $j\leq l$, we set $ a_j' = b_{j-1}'$, which is in fact the best we could do, and $b_j' = a_j'+ |I_j|= a_j'+ (b_j-a_j)$, thus $a_j' = a_1 + (b_1-a_1) + (b_2 - a_2) + \dots + (b_{j-1}-a_{j-1})$, and therefore the shift of $I_j$ is $a_j' - a_j = a_1 + (b_1-a_1) + (b_2 - a_2) + \dots + (b_{j-1}-a_{j-1}) - a_j = \sum_{i=1}^{j-1} (b_i-a_i) { -} (a_{i+1}-a_i)$. Applying now the inductive hypothesis to $I_{l+1}, \dots, I_{{ p}}$, we obtain ${ p}$ intervals $I_1', \dots, I_{{ p}}'$ satisfying the conditions in the lemma, and with: \[
\max_{1\leq i \leq { p}} |a_i'-a_i|  \leq
\max \left( \max_{1\leq j \leq l-1} \sum_{i=1}^{j} (b_i-a_i) { -} (a_{i+1}-a_i), 
 \max_{l+1\leq l_1 < l_2 \leq { p}-1} \left(  \sum_{i=l_1}^{l_2}  (b_i-a_i) - (a_{i+1}-a_{i})   \right)
 \right),
\] 
thus the lemma is proved. Furthermore, the upper bound is sharp when $l = { p}$. 
\end{proof}

\begin{lemma}\label{lem:step2vertexw}
Assume (\ref{eq:sizeVh2}) holds for all benchmarks $\bench{i} { <} \bench{j}$. 
Then it is sufficient to { further} change the weight of each vertex  in $V_{[0,{ 2}n)} \setminus (\yb{[0, { 2}n)}\cup \ys)$ by at most $\frac{{ 3n}}{\dd^{1/2-{  \alpha}}}$ in order to shift them to pairwise distinct values in $\zz'$ (provided $\dd$ and $n/\dd$ are sufficiently large).
\end{lemma}
\begin{proof}
For each integer $0 \leq h < 2n$, due to Claim \ref{claim:stepB21}, the weights of vertices in $V_h \setminus (\yb{[0,2n)}\cup \ys)$ form an interval of length $|V_h \setminus (\yb{[0,2n)}\cup \ys)|$ in $\zz'$ with the least element $a_h=\min(I_h\cap\zz')$. 
We now apply Lemma \ref{lem:integershift} to these intervals, taking into account only integers in $\zz'$, i.e., 
$(b_i-a_i)$ is evaluated as $|V_i \setminus (\yb{[0,2n)}\cup \ys)|$ and $a_{i+1}-a_i$ is substituted by $|I_i \cap \zz'|$ in the lemma.
Consequently, all the weights of vertices in $V_{[0,2n)} \setminus (\yb{[0,2n)}\cup \ys)$ may remain in $\zz'$ and get pairwise distinct  
via shifting each of the vertex weights by at most the following number of consecutive integers in $\zz'$:
%
%
%
\begin{align}
& \max_{{ 0}\leq l_1 < l_2 \leq { 2n-2}} \left(  \sum_{i=l_1}^{l_2}  |V_i \setminus (\yb{[0,2n)}\cup \ys)| - |I_i \cap \zz'| \right)  \leq \max_{{ 0}\leq l_1 < l_2 \leq { 2n-2}} \left(  \sum_{i=l_1}^{l_2}  |V_i| - |I_i \cap \zz'| \right) \nonumber \\
\leq & \max_{{ 0}\leq l_1 < l_2 \leq { 2n-2}} \left(  \sum_{i=l_1}^{l_2}  |V_i| - (|I_i|-{ 6}) (k-6)/k \right) 
\leq  \max_{{ 0}\leq l_1 < l_2 \leq { 2n-2}} \left(  \sum_{i=l_1}^{l_2}  \left(|V_i| - (\ceiling{n/\dd}+a'-{ 7}) (k-6)/k \right)\right),
 \label{eq:shiftzz'}
\end{align}
where the second inequality follows by (\ref{eq:countnotk}) and the last inequality uses $|I_i| = \ceiling{n/\dd}+a'-1$.

In order to upper-bound the quantity in (\ref{eq:shiftzz'}), suppose the maximization is achieved when $l_1 = h_1$ and $l_2 = h_2$, and suppose $h_1$ is between benchmarks $\bench{i-1}$ and $\bench{i}$, i.e., $\bench{i-1} { \leq} h_1 { <} \bench{i}$. Similarly, suppose $\bench{j} \leq h_2 < \bench{j+1}$ for some $j\geq i-1$. Thus the last quantity in (\ref{eq:shiftzz'}) equals 
\begin{align}
& \sum_{h\in [h_1, h_2]} \left(|V_h| - (\ceiling{n/\dd}+a'-{ 7}) (k-6)/k\right) \nonumber \\
{ \leq}  & \sum_{h\in [\bench{i}, \bench{j})} \left(|V_h| - (\ceiling{n/\dd}+a'-{ 7}) (k-6)/k\right) + \sum_{h\in [\bench{j}, h_2+1)} \left(|V_h| - (\ceiling{n/\dd}+a'-{ 7}) (k-6)/k\right) \nonumber \\
& + \sum_{h\in [h_1,\bench{i})} \left(|V_h| - (\ceiling{n/\dd}+a'-{ 7}) (k-6)/k\right) \nonumber \\
\leq & \sum_{h\in [\bench{i}, \bench{j})} \left(|V_h| - (\ceiling{n/\dd}+a'-{ 7}) (k-6)/k\right) + \sum_{h\in [\bench{j}, h_2+1)} |V_h|+ \sum_{h\in [h_1,\bench{i})} |V_h| \nonumber \\
{ \leq}   & \sum_{h\in [\bench{i}, \bench{j})} \left(|V_h| - (\ceiling{n/\dd}+a'-{ 7}) (k-6)/k\right) + \sum_{h\in [\bench{j}, \bench{j+1})} |V_h|+ \sum_{h\in [\bench{i-1},\bench{i})} |V_h|. 
 \label{eq:1}
\end{align}
By (\ref{eq:sizeVh2}) and  (\ref{eq:bench}) (implying that $\mu_{[\bench{t}, \bench{t+1})}\leq { 2n/\dd}$ for every given $t$),
\begin{equation}
\sum_{h\in [\bench{j}, \bench{j+1})} |V_h|+ \sum_{h\in [\bench{i-1},\bench{i})} |V_h|  \leq 
\mu_{[\bench{i-1}, \bench{i})} + \mu_{[\bench{j}, \bench{j+1})} 
+\sqrt{\frac{n \mu_{[\bench{i-1}, \bench{i})}}{\dd^{1-2{  \alpha}}}}
+\sqrt{\frac{n \mu_{[\bench{j}, \bench{j+1})}}{\dd^{1-2{  \alpha}}}} 
<3 n/\dd^{1-{  \alpha}}. \label{SideVhBound}
\end{equation}
Analogously, by (\ref{eq:sizeVh2}) and the facts that $\mu_{[\bench{i}, \bench{j})} \leq n$ and
$\mu_h \leq n/\dd$ for each $h$ due to Claim \ref{claim:smvi}, 
\begin{equation}
\sum_{h\in [\bench{i}, \bench{j})} |V_h| \leq \mu_{[\bench{i}, \bench{j})}  +\sqrt{\frac{n \mu_{[\bench{i}, \bench{j})}}{\dd^{1-2{  \alpha}}}}
\leq
 \frac{n(\bench{j}-\bench{i})}{\dd} + \frac{n}{\dd^{1/2-{  \alpha}}}. \label{MostVhBound}
\end{equation}
By~(\ref{SideVhBound}) and~(\ref{MostVhBound}), the last quantity in~ (\ref{eq:1}) is thus
bounded above by 
\begin{align*}
& \frac{n(\bench{j}-\bench{i})}{\dd} + \frac{{ n}}{\dd^{1/2-{  \alpha}}} - (\bench{j}-\bench{i})(\ceiling{n/\dd}+a'-{ 7}) \frac{k-6}{k} +  \frac{3n}{\dd^{1-{  \alpha}}} \nonumber \\
\leq & \frac{n(\bench{j}-\bench{i})}{\dd} + \frac{{ n}}{\dd^{1/2-{  \alpha}}}- (\bench{j}-\bench{i})\frac{n}{\dd}\frac{k-6}{k} - (\bench{j}-\bench{i})(a'-{ 7})\frac{k-6}{k} + \frac{3n}{\dd^{1-{  \alpha}}} \nonumber \\
= & \frac{n(\bench{j}-\bench{i})}{\dd} + \frac{{ n}}{\dd^{1/2-{  \alpha}}} - (\bench{j}-\bench{i})\frac{n}{\dd} +  (\bench{j}-\bench{i})\frac{n}{\dd}\frac{6}{k}    - (\bench{j}-\bench{i})(a'-{ 7})\frac{k-6}{k} + \frac{3n}{\dd^{1-{  \alpha}}} \nonumber \\
= & \frac{{ n}}{\dd^{1/2-{  \alpha}}}+  (\bench{j}-\bench{i}) \left(\frac{n}{\dd}\frac{6}{k}    - (a'-{ 7})\frac{k-6}{k} \right)+ \frac{3n}{\dd^{1-{  \alpha}}}  \leq \frac{{ n}}{\dd^{1/2-{  \alpha}}} + \frac{3n}{\dd^{1-{  \alpha}}} < \frac{{ 2n}}{\dd^{1/2-{  \alpha}}} .
\end{align*}
The second to last inequality above holds because by the definitions of $k$ and $a'$, \[\frac{n}{\dd}\frac{6}{k}    - (a'-{ 7})\frac{k-6}{k} \leq 0\] for $\delta$ and $n/\delta$ (thus also $k$) large enough.
We have thus proved that each vertex needs to shift its weight by at most $\frac{{ 2n}}{\dd^{1/2-{  \alpha}}} $ consecutive integers in $\zz'$. Hence,  by (\ref{eq:countnotk2}), each vertex needs to shift its weight by at most $\left(\left(\frac{{ 2n}}{\dd^{1/2-{  \alpha}}} { +1}\right)\frac{k}{k-6}+{ 6}\right){ -1} < \frac{{ 3n}}{\dd^{1/2-{  \alpha}}}$ (in $\zz$). 
\end{proof}

\subsubsection{Summary of Step B.} 
\begin{corollary}\label{lem:step2edgew}
Assume (\ref{eq:sizeVh2}) holds for all benchmarks $\bench{i} < \bench{j}$.
Then the following can be achieved (for $\delta$ and $n/\delta$ large enough).

Vertices $v$ in $B \setminus \left( \yb{[0,{ 2}n)} \cup \ys  \right)$ can be provided 
distinct weights in $\zz'$ due to appropriately chosen $f_2$ supported on edges across $B \setminus \left( \yb{[0,{ 2}n)} \cup \ys  \right)$ and $S$ with $\|f_2\|_\infty \leq \frac{{ 11}n}{\dd^{1+\epsilon}}+2$. 

Consequently, each edge $e$ in $E(B)$ satisfies $(f_1+f_2)(e) = 1$ or $\ceiling{n/\dd}+a'$, while each edge $e$ between $B$ and $S_q$, for $1\leq q \leq k'$, satisfies 
\begin{equation}
(f_1+f_2)(e) \in \left[\ceiling{\ceiling{\frac{n}{\dd}}\frac{1}{3k'}}(k' +q)- \left(\frac{{  11}n}{\dd^{1+\epsilon}}+2\right),\ceiling{\ceiling{\frac{n}{\dd}}\frac{1}{3k'}}(k'+q)+ \left(\frac{{  11}n}{\dd^{1+\epsilon}}+2\right)\right]. \label{eq:f1f2intBS}
\end{equation} 
\end{corollary}
\begin{proof} 
Claim \ref{claim:stepB1} within Step 1 shows
we may first modify weights of edges across $B$ and $S$ by at most  ${  2}\left( \ceiling{\frac{n}{\dd}}+a'\right)\dd^{1/2+{  \alpha}} /\sstar +1$  so that 
for each $v \in B \setminus \left( \yb{[0,{  2}n)} \cup \ys  \right)$, $(\fv_1+\fv_2)(v)$ is {   at most one away from } $\sm(v,Z(v))$. 

In Step 2-1, Claim \ref{claim:stepB21} exposes that 
the weight of each vertex $v\in V_h \setminus (\yb{h}\cup \ys)\subseteq V_{[\bench{i}, \bench{i+1})}$ (for any given $h,i$) 
needs to further change by at most $\max(|V_h|, |I_h|)k/(k-6)+ {  12}$ to be shifted to $\zz'$ and get distinguished from the remaining ones in $V_h \setminus (\yb{h}\cup \ys)$. 
By~(\ref{eq:sizeVh2}) and~(\ref{eq:bench}) 
we have 
$|V_h| \leq  \mu_{[\bench{i}, \bench{i+1})}  +\sqrt{\frac{n \mu_{[\bench{i}, \bench{i+1})}}{\dd^{1-2{  \alpha}}}} 
 \leq {  2n/\dd} + {  \sqrt{2}}n/\dd^{1-{  \alpha}}  < {  1.5} n/\dd^{1-{  \alpha}}$. 
We also have $|I_h| = \ceiling{n/\dd}+a'-1 < {  1.5} n/\dd^{1-{  \alpha}}$. Thus the weight of $v$ needs to shift by at most
$
( {  1.5} n/\dd^{1-{  \alpha}}) \cdot k/(k-6)+ {  12} <  {  2}n/\dd^{1-{  \alpha}}
$
within this step.

Finally, within Step 2-2, by Lemma \ref{lem:step2vertexw}, the vertices in $V_{[0,{  2}n)} \setminus (\yb{[0, {  2}n)}\cup \ys)$ need to change their weights by at most $ {  3}n/\dd^{1/2-{  \alpha}} $ to make them pairwise distinct and keep them in $\zz'$. 

Steps 2-1 and 2-2 together require changing 
vertex weights by at most ${  2}n/\dd^{1-{  \alpha}} + {  3}n/\dd^{1/2-{  \alpha}} < 
{  3.5}n/\dd^{1/2-{  \alpha}}$. 
Since $v \notin \ys$, it has at least $\sstar/2$ neighbors in $S$. Therefore we only need to modify the weight of each edge between $v$ and $S$ by at most
$
\ceiling{({  3.5}n/\dd^{1/2-{  \alpha}}) /(\sstar/2)} < {  7}n/\dd^{1+\epsilon}+1
$
in Step 2.

In Steps 1 and 2 combined, the weight of each edge across $B \setminus \left( \yb{[0,{  2}n)} \cup \ys  \right)$ and $S$ is thus changed by at most 
\[
{  2}\left( \ceiling{\frac{n}{\dd}}+a'\right)\frac{\dd^{\frac12+{  \alpha}}}{\sstar}  +1 + \frac{{  7}n}{\dd^{1+\epsilon}}+1 < {  4}\left( {\frac{n}{\dd}}\right) \frac{\dd^{\frac12+{  \alpha}}}{\sstar} +1 + \frac{{  7}n}{\dd^{1+\epsilon}}+1 
{  \leq} \frac{{  11}n}{\dd^{1+\epsilon}}+2.
\]

The lemma is thus proved, as its last statement follows by the definition of $f_1$ in Step A.
\end{proof}


\subsection{Preparations to Step C.}
We define the last set of bad vertices. 
\begin{definition}
For each integer $1 \leq q \leq k'$, let $\ysq{q}$ be the set of vertices $v$ in $S_q$ such that $\left|\deg_B(v) - \deg(v)(\dd-\sstar)/\dd \right| {>} \deg(v)^{1/2+\epsilon}$. 
\end{definition}
\begin{lemma}\label{lem:ysq}
With probability at least $1 
- k'e^{-\dd^{2\epsilon}/6}$,
for all $1 \leq q \leq k'$.
\begin{align}
& |\ysq{q}| {<} \frac{4n\sstar}{\dd k'} \exp(-\dd^{2\epsilon}/6), \label{ZSQBound} 
\end{align}
In addition, when $\dd > \sqrt{n}$,  with probability at least  $1- 1/{n}$, $\bigcup_{1\leq q \leq k'} \ysq{q} = \emptyset$.
\end{lemma}
\begin{proof}
Consider any fixed integer $q\in [1,k']$.
For each vertex $v$, let 
$Z_v$ be the indicator random variable that $v \in S_q$ but $\left|\deg_B(v) - \deg(v)(\dd-\sstar)/\dd \right|$ $ {>}   \deg(v)^{1/2+\epsilon}$.
Thus 
$|\ysq{q}| = \sum_v Z_v$. 
For each vertex $v$,
\[
\Pr(Z_v = 1) = \Pr\left(v\in S_q\right) \Pr\left(\left|\deg_B(v) - \deg(v)(\dd-\sstar)/\dd \right|  {>}   \deg(v)^{1/2+\epsilon} \given[\Big] v\in S_q \right).
\]
Each neighbor of $v$ is placed in $B$ independently with probability $(\dd-\sstar)/\dd$, and thus $\ee[\deg_B(v)]= \deg(v)(\dd-\sstar)/\dd$, which is greater than $\deg(v)^{1/2+\epsilon}$. Thus by the Chernoff Bound, 
\begin{align}
& \Pr\left(\left|\deg_B(v) - \deg(v)\frac{\dd-\sstar}{\dd} \right|  {>}   \deg(v)^{1/2+\epsilon} \given[\Big] v\in S_q \right) 
 \leq  2\exp\left(- \frac{\deg(v)^{1+2\epsilon}}{3\deg(v)\frac{\dd-\sstar}{\dd}}\right) < 2e^{-\dd^{2\epsilon}/3}. \label{eq:przv}
\end{align}
Therefore, 
\begin{equation}
\ee[|\ysq{q}|] =\ee\left[\sum_v Z_v\right] {   = \sum_v \Pr(Z_v = 1)}
\leq  \sum_v \Pr(v\in S_q)\cdot  2\exp(-\dd^{2\epsilon}/3) < \frac{2n\sstar}{\dd k'}\cdot  2\exp(-\dd^{2\epsilon}/3). \label{eq:eeysq}
\end{equation}
By Markov's Inequality and (\ref{eq:eeysq}),
\[
\Pr\left(|\ysq{q}| \geq \frac{2n\sstar}{\dd k'}\cdot  2\exp(-\dd^{2\epsilon}/6)\right) \leq \ee[|\ysq{q}|]/\left(\frac{2n\sstar}{\dd k'}\cdot  2\exp(-\dd^{2\epsilon}/6)\right) \leq \exp(-\dd^{2\epsilon}/6),
\]
and therefore, (\ref{ZSQBound}) holds by a union bound for $1 \leq q \leq k'$.

In addition, when $\dd > \sqrt{n}$, by a union bound over $v$ and $q$ on (\ref{eq:przv}), with probability at least $1-nk' 2e^{-\dd^{2\epsilon}/3} > 1-1/{  n}$, $\ysq{q} = \emptyset$ for all $1 \leq q \leq k'$. 
\end{proof}

So far we have obtained the following sets of bad vertices, which require different treatments in Step C: 
\begin{enumerate}
\item the set $\yb{[0,2n)} \subset B$ whose weights in Step A are not close to the expected values, 
\item the set $\ys$ whose degrees to $S$ are {   less than} $\sstar {   \deg(v)}/(2\dd)$,   
\item a set $\ysn \subset B$ which are neighbors of vertices in $\ys$ such that each vertex in $\ys$ has at least $\sstar/2$ neighbors in $\ysn$, and lastly,
\item the sets $\ysq{q} \subset S_q$, 
$1\leq q \leq k'$ of vertices whose degrees to $B$ are far from expected. 
\end{enumerate} 
\begin{definition}
Denote by $\Ub$ the union of all these four types of 
{\it bad vertices}.  
The complement of $\Ub$ in $S$ will in turn be referred to as the set of {\it good vertices} and denoted $\Ug$, $\Ug =S \setminus \Ub \subset S$.
\end{definition}

\begin{claim}\label{claim:sizeUb}
With probability at least $1- {  2}e^{-n/({  4}\dd^{0.5})}
 - ({  8}\dd^2)e^{-\frac{1}{{  3}}\sqrt{\frac{n}{\dd^{1-2{  \alpha}}}}} 
 - {  5\dd^2} 
 e^{-\dd^{{  2}{  \alpha}}/{  4}}$, 
\[|\Ub| <  {  3}{  n}e^{{-\dd^{ {  2}{  \alpha}}/ {  4}}} 
\]
(provided $\dd$ and $n/\dd$ are sufficiently large).
Furthermore, if $\dd > \sqrt{n}$, then with probability at least $1- {  3}/{  n}$, $\Ub = \emptyset$. 
\end{claim}
\begin{proof}
By Corollary \ref{cor:yb}, Lemma \ref{lem:ysysn} and Lemma \ref{lem:ysq}, 
\begin{align*}
|\Ub| {  <} & |\yb{[0,{  2}n)}| + |\ys| + |\ysn| + \left(\sum_{q=1}^{k'} |\ysq{q}| \right) \\
\leq & {  2}n e^{{-\dd^{{  2}{  \alpha}}/{  4}}} + 2ne^{-\sstar/24} +{  2}n \sstar e^{-\sstar/24} + k' \frac{4n\sstar}{\dd k'} e^{-\dd^{2\epsilon}/6}  
<   {  3}{  n}e^{{-\dd^{ {  2}{  \alpha}}/ {  4}}}
\end{align*}
with probability at least $1
- {  4}\dd^2e^{-\dd^{{  2}{  \alpha}}/{  4}}
 - ({  8}\dd^2)e^{-\frac{1}{{  3}}\sqrt{\frac{n}{\dd^{1-2{  \alpha}}}}} 
- e^{-\sstar/24} 
-  2e^{- n/({  4}\delta^{0.5})} 
- k'e^{-\dd^{2\epsilon}/6} 
 > 1- {  2}e^{-n/({  4}\dd^{0.5})}
 - ({  8}\dd^2)e^{-\frac{1}{{  3}}\sqrt{\frac{n}{\dd^{1-2{  \alpha}}}}} 
 - {  5\dd^2} 
 e^{-\dd^{{2}{  \alpha}}/{  4}}$.

In addition, when $\dd > \sqrt{n}$, again by  Corollary \ref{cor:yb}, Lemma \ref{lem:ysysn} and Lemma \ref{lem:ysq}, 
with probability at least $1- {  3}/{  n}$, $\Ub = \emptyset$. 
\end{proof}

We will next define two vertices in $S$ being ``close" (Definition \ref{def:uvclose}) if their vertex weights could potentially be very close in terms of values after Step C.  
We will make sure vertices which are ``close" do not have the same vertex weight after Step C. Informally, 
$u \in L(v)$ if $u, v$ are ``close", whereas for $u \notin L(v)$, we will prove later that the weights of $u,v$ cannot be the same after Step C. 

\begin{definition}\label{def:uvclose}
For any two vertices $v, u\in S$, we
say $u\in L(v)$ and $v \in L(u)$ if there are integers $1 \leq p, q \leq k'$ such that $v\in S_q$, $u \in S_p$ and both of the following two conditions hold.
\begin{align}
\left(\ceiling{\ceiling{\frac{n}{\dd}}\frac{1}{3k'}}(k'+q)  + \left(\frac{{   13}n}{\dd^{1+\epsilon}} + 4 \right)  
\right)\deg(v) >
\left(\ceiling{\ceiling{\frac{n}{\dd}}\frac{1}{3k'}}(k'+p)  - \left(\frac{{   13}n}{\dd^{1+\epsilon}} + 4 \right)  
\right)\deg(u); \label{eq:L1} \\
\left(\ceiling{\ceiling{\frac{n}{\dd}}\frac{1}{3k'}}(k'+q)  - \left(\frac{{   13}n}{\dd^{1+\epsilon}} + 4 \right)  
\right)\deg(v) <
\left(\ceiling{\ceiling{\frac{n}{\dd}}\frac{1}{3k'}}(k'+p)  + \left(\frac{{   13}n}{\dd^{1+\epsilon}} + 4 \right)  
\right)\deg(u).\label{eq:L2} 
\end{align}
\end{definition}

\begin{lemma}\label{lem:closeH}
With probability at least $1- {   2n} 
\exp\left(- \frac{1}{3}{\frac{\sstar n}{\dd k'}}\right)$ (if $n/\dd$ and $\dd$ are large enough), 
for every $v \in S$, 
\[  |L(v)|  \leq 2{   \ceiling{\frac{\sstar}{k'}}\frac{n}{\dd}}. 
\]
\end{lemma}
\begin{proof}
Suppose (\ref{eq:L1}) holds. Note $\ceiling{\ceiling{n/\dd}/(3k')}(k'+q) \leq \ceiling{\ceiling{n/\dd}/(3k')}(k'+k') < (((n/\dd)+1)/(3k')+1)2k' = {   2n/(3\dd) + 2/3 + 2k' < 0.99 n/\dd}$. Therefore the {   left hand side of (\ref{eq:L1}) equals at most $(n/\dd) \deg(v)$}. Analogously, the right hand side of (\ref{eq:L1}) equals at least $\left(n/(3\dd) - n/(6\dd)\right)\deg(u) = \deg(u)n/(6\dd)$.
Therefore (\ref{eq:L1})  implies: 
$(n/\dd) \deg(v) > \deg(u)n/(6\dd)$, i.e.,
$
6\deg(v) > \deg(u).
$
Similarly, (\ref{eq:L2}) implies $
6\deg(u) > \deg(v).
$ Therefore if $u\in L(v)$, then 
\begin{equation}
1/6< \deg(v)/\deg(u)< 6. \label{eq:bounduv}
\end{equation}

Given $v\in S_q$ for some fixed $1\leq q \leq k'$, we compute the probability that $|L(v)|$ is large. 
To this end we first show that 
for a given vertex $u\in V(G)\setminus \{v\}$, if it satisfies both (\ref{eq:L1}) and (\ref{eq:L2}), then there
is only one $S_p$ with $1\leq p \leq k'$ that $u$ can be placed in.
Rearranging inequality (\ref{eq:L1}), we obtain
\begin{align*}
k'+p < \frac{\left(\ceiling{\ceiling{n/\dd}/(3k')}(k'+q)  + \left(\frac{{   13}n}{\dd^{1+\epsilon}} + 4 \right)  
\right)\frac{\deg(v)}{\deg(u)} + \left(\frac{{   13}n}{\dd^{1+\epsilon}} + 4 \right)}{\ceiling{\ceiling{n/\dd}/(3k')}}.
\end{align*}
Similarly by (\ref{eq:L2}), 
\begin{align*}
k'+p > \frac{\left(\ceiling{\ceiling{n/\dd}/(3k')}(k'+q)  - \left(\frac{{   13}n}{\dd^{1+\epsilon}} + 4 \right)  
\right)\frac{\deg(v)}{\deg(u)} - \left(\frac{{   13}n}{\dd^{1+\epsilon}} + 4 \right)}{\ceiling{\ceiling{n/\dd}/(3k')}}.
\end{align*}
These two inequalities mean that if $u \in L(v)$, then by~(\ref{eq:bounduv}),  $k'+{   p}$ must belong to 
an interval in $\mathbb{R}$ of length at most
\begin{align*}
& 2\left(\left(\frac{{   13}n}{\dd^{1+\epsilon}} + 4 \right) \cdot 6 +\left(\frac{{   13}n}{\dd^{1+\epsilon}} + 4 \right) \right)/ \ceiling{\ceiling{n/\dd}/(3k')}  \\
< & 14 \left(\frac{{   13}n}{\dd^{1+\epsilon}} + 4 \right) / (n/(3\dd k')) 
<  {   \frac{{   546} k'}{\dd^{\epsilon}} + \frac{168 k' \dd}{n} < 1}.
\end{align*}
Therefore there indeed may be at most one ${   p}$ so that $u \in S_{{   p}}$ implies $u \in L(v)$. Thus, 
\begin{equation}
\Pr\left(u \in L(v)  {    \given[\big] v \in S_q } \right) \leq {   \ceiling{\frac{\sstar}{k'}}\frac{1}{\dd}}, \nonumber
\label{PruInLvBound}
\end{equation}
and hence 
the expected number of vertices $u$ such that $u \in L(v)$ {   given $v\in S_q$} equals at most 
${   \ceiling{\sstar/k'}  (n-1)/\dd < \ceiling{\sstar/k'}  n/\dd}$.
Since for fixed $v$ and $q$ such that $v\in S_q$, 
the events $u \in L(v)$ are independent for all $u{   \neq v}$, 
by the Chernoff Bound 
we thus obtain that:
\begin{align*}
 {   \Pr\left(|L(v)| - \ceiling{\frac{\sstar}{k'}}\frac{n}{\dd}  > \ceiling{\frac{\sstar}{k'}}\frac{n}{\dd} \given[\Big] v \in S_q \right)} & 
{   \leq 2 \exp\left(-\frac{1}{3} \left(\ceiling{\frac{\sstar}{k'}}\frac{n}{\dd}\right)\right)}
 {   \leq} 2 \exp\left(-\frac{1}{3} \left({\frac{\sstar n}{\dd k'}}\right)\right) .
\end{align*}
Therefore, by the law of total probability, for any given $v\in V(G)$,
${    \Pr\left(|L(v)|  > 2\ceiling{\frac{\sstar}{k'}}\frac{n}{\dd} \given[\Big] v \in S \right)} $ 
${    \leq}$ 
${    2 \exp\left(-\frac{1}{3} \left({\frac{\sstar n}{\dd k'}}\right)\right),}$
 and hence,
\begin{align*}
{   \Pr\left( \left(v \in S\right) \Rightarrow \left(|L(v)|  \leq 2\ceiling{\frac{\sstar}{k'}}\frac{n}{\dd} \right) \right)~ } 
&{   =} 
{   ~1- \Pr\left( \left(v \in S\right) \wedge \left(|L(v)|  > 2\ceiling{\frac{\sstar}{k'}}\frac{n}{\dd} \right) \right)  }\\
& {   \geq} 
{   ~1- \Pr\left(|L(v)|  > 2\ceiling{\frac{\sstar}{k'}}\frac{n}{\dd} \given[\Big] v \in S \right)  
 \geq 
 1- 2 \exp\left(-\frac{1}{3} \left({\frac{\sstar n}{\dd k'}}\right)\right).}
\end{align*}
By  a union bound over all vertices $v$ we thus obtain the thesis.
\end{proof}

\begin{corollary}\label{cor:sizeUbLv}
With positive probability, all the following inequalities hold (for $\dd$ and $n/\dd$ sufficiently large): 
\begin{align*}
& |\Ub| < {   3} {   n}\exp({-\dd^{ {   2}{   \alpha}}}/ {   4}) , \\
&|S| \in  [\sstar n/\dd - n/\delta^{0.5-\epsilon}, \sstar n/\dd + n/\delta^{0.5-\epsilon}], \\
& |L(v)|  \leq 2{   \ceiling{\frac{\sstar}{k'}}\frac{n}{\dd}} \text{ for all } v\in S.\\
&  {    \left|V_{[\bench{i}, \bench{j})}\right|  \leq  \mu_{[\bench{i}, \bench{j})}  +\sqrt{\frac{n \mu_{[\bench{i}, \bench{j})}}{\dd^{1-2{   \alpha}}}}}
\text{ for any two benchmarks } \bench{i}{   <} \bench{j}. 
\end{align*}
In addition, when $\dd > \sqrt{n}$, 
then $\Ub = \emptyset$. 
\end{corollary}
\begin{proof}
This corollary is an immediate consequence of  Corollary \ref{cor:yb},  Lemma~\ref{lem:ysysn}
Claim \ref{claim:sizeUb} and Lemma \ref{lem:closeH}, as 
$1-{   4}\dd^2 {   e^{-\dd^{2{   \alpha}}/4}} 
- {   8}\dd^2\exp\left(-\frac{1}{{   3}}\sqrt{\frac{n}{\dd^{1-2{   \alpha}}}}\right)
- e^{-\sstar/24}
- 2e^{- n/({   4}\delta^{0.5})}
- {   2}e^{-n/({   4}\dd^{0.5})}
- {   8}\dd^2\exp\left(-\frac{1}{{   3}}\sqrt{\frac{n}{\dd^{1-2{   \alpha}}}}\right)
 - {   5\dd^2} e^{-\dd^{{   2}{   \alpha}}/{   4}}
 - {   2n} \exp\left(- \frac{1}{3}{\frac{\sstar n}{\dd k'}}\right)
 -\frac{{   3}}{{   n}}
 > 0
$
when $\dd$ is sufficiently large.
\end{proof}

\subsection{Step C.}
Throughout Step C, we assume the statements in Corollary \ref{cor:sizeUbLv} hold. 
\subsubsection{Goal.} 
\begin{definition}
Let $G'$ be the graph on the vertex set $\Ub \cup S$ whose edges consist of all the edges in $G[S]$, all the edges between $\Ub$ and $S$ in $G$ and all the edges between $\ys$ and $\ysn \subset B$ in $G$. 
\end{definition}
Note that by the definition of $G'$, $\ys$ and $\ysn$, for every $v\in V(G')$, 
\begin{equation}\label{dG'Bound}
\deg_{G'}(v)\geq \frac{\sstar}{2}.
\end{equation}

In this step we will only change the weights of edges in $G'$.  
Our goal is to obtain pairwise distinct weights in $\zz \setminus \zz'$ (i.e., equal to $0,1, \dots, 5$ modulo $k$) for all vertices in $V(G') = \Ub \cup S$ after Step C.
Within it we will 
not change the weights of edges incident to vertices in $B \setminus \Ub$. Therefore the weights of vertices in $B \setminus \Ub$ will remain distinct and in $\zz'$ by Steps A and B (Corollary~\ref{lem:step2edgew}). 


\subsubsection{Step C-1.}
\paragraph{Initialization.}
We initialize Step C by assigning to all edges in $G'[S]=G[S]$ and $G'[B]$ the new weights: $\ceiling{\ceiling{\frac{n}{\dd}}/2}
$. We do not modify the weights of edges across $B$ and $S$ in $G'$ yet, though. 

Suppose $C_1, \dots, C_T$ are the non-trivial connected components in $G'[S]$ (i.e., of order larger than one), ordered arbitrarily,  and let $W$ be the set of isolated vertices in $G'[S]$. Clearly $W \subset \ys \subset \Ub$ by the definitions of $\ys$ and $\Ub$. 
Set \[U': = V(G') \setminus (S\setminus W).\]

\begin{claim}\label{claim:stepC1}
We may modify every edge weight in $G'$ by at most $2$ so that 
each vertex weight in $U'$ equals $0$ or $1$ modulo $k$, 
each vertex weight in $\Ug \setminus U'$ equals $2$ or $3$ modulo $k$ and 
each vertex weight in $\Ub \setminus U'$ equals $4$ or $5$ modulo $k$.
\end{claim}
\begin{proof}
We will apply an algorithm analogous to the ones used in~\cite{KKP,Pasy}, whose origins date back to~\cite{Kthesis}.

Given an arbitrary ordering $v_1, v_2,  \dots, v_{|S \cup \Ub|}$ of vertices in $G'$, each edge $\{v_i, v_j\}$ with $i<j$ is called a {\it forward edge} of $v_i$ and a  {\it backward edge} of $v_j$. 
For each $v \in U'$, define a set {   $\ssett_{v} = \{0,1\}$},  for each $v \in \Ug\setminus U'$, define  {   $\ssett_{v} =\{2,3\}$}, and for each $v \in \Ub\setminus U'$, define  {   $\ssett_{v} =\{4,5\}$}. 
 For $i = 1, 2, \dots$, we consider each consecutive $v_i$ after another and modify 
weights of edges  incident with $v_i$ in $G'$ so that the weight of $v_i$ lands in {   $\ssett_{v}$} modulo $k$.
We also guarantee that this weight does not leave  {   $\ssett_{v}$} (modulo $k$) throughout the further part of the algorithm.

By~(\ref{dG'Bound}), the first vertex $v_1$ has at least $\sstar/2$ neighbors, 
i.e., at least $\sstar/2$ forward edges in $G'$.
We modify each forward edge of $v_1$ by adding $0$ or $1$ to its weight. 
Thereby, we may obtain at least ${   \sstar/2}{   +1} {   > k + 6}$ distinct weights which are consecutive integers for $v_1$.
Thus there is a way of choosing these modifications so that the weight of $v_1$ belongs to  {   $\ssett_{v}$} modulo $k$. 

We then proceed consecutively with $v_i$, $i=2,3,\ldots$. 
For a given $i$, we again admit adding $0$ or $1$ to the weights of the forward edges of $v_i$. 
There are two admissible modifications for the weights of backward edges of $v_i$ as well. These belong to the set $\{-1,0,1\}$.
Specifically, say $\{v_i,u\}$ is a backward edge of $v_i$. If the vertex weight of $u$  modulo $k$ is currently the smaller value in {   $\ssett_{u}$}, then we admit modifying the weight of $\{v_i,u\}$ by adding $0$ or $1$. Note that this in particular guarantees that the updated vertex weight of $u$ will remain in {   $\ssett_{u}$} modulo $k$. By the same reason, we admit adding $0$ or $-1$ to the weight of $\{v_i,u\}$ if the vertex weight of $u$  modulo $k$ is currently the larger value in {   $\ssett_{u}$}.

Consequently, analogously as for $v_1$, we may thereby obtain at least  $\deg_{G'}(v) {   +1} \geq \sstar/2 {   +1} > k+6$ weights which are consecutive integers for $v_i$.
Thus there is a way of choosing these modifications so that the weight of $v_i$ belongs to  {   $\ssett_{v}$} modulo $k$, which is our goal.


Since each edge in $G'$ can be modified at most twice: once as a forward edge and once as a backward edge, each edge in $G'$ changes its weight by at most $2$ in Step C-1. 
\end{proof}

\subsubsection{Step C-2.}

Step C-2 is more technical. We in particular handle all bad vertices within it. 
Given an ordering $v_1, \dots, v_{|S\cup \Ub|}$ of vertices in $G'$ (specified later), again each edge $\{v_i, v_j\}$ with $i<j$ is called a {\it forward edge} of $v_i$ and a  {\it backward edge} of $v_j$. 
We will again use an algorithm inspired by~\cite{Kthesis,KKP,Pasy}.

If $v_i \in \Ub$, we say all its forward and backward edges are {\it active}. If $v_i\in \Ug$, then only its forward and backward edges in $E(S)$ are called {\it active}. (Note that a good vertex $v$ in $S$ could  be adjacent to some vertex $u$ in $\Ub\setminus S$. Such an edge would still be active for $u\in \Ub$, but would not be active for $v\in \Ug$.) 
We call a vertex {\it terminal} if it has no active forward edge in the ordering.

Recall that by the definitions in Step C-1, the non-trivial components in $G'[S]$: 
 $C_1,\dots, C_T$ together with the set $W$ of isolated vertices in $S$ partition 
 $S$.
 Furthermore, as
$W \subset \ys \subset \Ub$, by Corollary \ref{cor:sizeUbLv}, $|W| \leq |\Ub| < |S|$, and thus there is at least one nontrivial connected component in $G'[S]$. 
Recall \[V(G') = S\cup \Ub, \ \ \ U' = V(G') \setminus (C_1 \cup \dots \cup C_T).\] 
The following properties of $G'$ are immediate from the definitions of different types 
of bad vertices and $G'$. 
\begin{claim}\label{claim:G'}
Each vertex in $G'$ has at least $\sstar/2$ active edges in $G'$, 
\begin{equation}
U' \subset \Ub, \ \  \Ug \subset S \setminus U',  \ \  V(G') = U'\cup S. \label{eq:U'}
\end{equation}
Thus, in particular, 
if $v \in V(G')$ is a good vertex, then its neighbors {in $G'$} which are not in $S$ are bad vertices. 
\end{claim}

\paragraph{Ordering of vertices.} 
We now specify the ordering of the vertices in $G'$. 
Let $t_i, i=1,2,\dots$ denote the terminal vertices in the ordering. For each $i$, let $r_i$ be the vertex immediately preceding    
$t_i$ in the ordering.
We show there is an ordering satisfying the following claim. 
\begin{claim}\label{claim:ordering}
There is an ordering  of the vertices in $G'$ such that all vertices in $U'$ come before vertices not in $U'$. 
Furthermore, 
 the ordering satisfies the following conditions. 
\begin{enumerate}
\item For each $i$, $\{r_i, t_i\}$ is always an edge in $G'$ and moreover, 
this edge is an active forward edge for $r_i$.
\item  The vertex sets $\{r_i, t_i\}_i$ are pairwise disjoint.
\item For each $i$,  the pair $\{r_i, t_i\}$ satisfies one of the following: 
either both $t_i, r_i$ are in $U' \subset \Ub$, or both $t_i, r_i$ are in $S \setminus U'$. 
\end{enumerate}
\end{claim}
\begin{proof}
Let $C_1', \dots, C_{T'}'$ be the non-trivial connected components in $G'[U']$. Suppose $W'$ is the set of isolated vertices in $G'[U']$. 

We order the vertices in $G'$ as follows: we start from 
the vertices in $W'$, ordered arbitrarily. 
Coming next will be the vertices in $C_1', \dots, C_{T'}'$ sequentially. Last in the ordering will in turn be vertices in $C_1, \dots, C_T$ sequentially. Note that we have not specified the ordering within each $C_i$ or $C_i'$ yet. 
Nevertheless, it is already clear that
vertices in $U'$ will all come before vertices not in $U'$ in the ordering. (See Figure \ref{fig:1}.)

\begin{center}
\begin{figure}[h]
\includegraphics[scale=0.6]{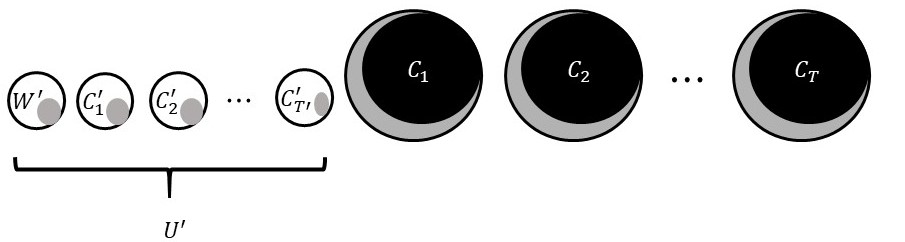}
\caption{Illustration of ordering of vertices in $V(G')$. \small{Black vertices are good vertices in $S$; gray vertices are bad vertices in $S$, and white vertices are bad vertices in $U'\setminus S$.} }
\label{fig:1}
\end{figure}
\end{center}

We finally specify the 
ordering within
non-trivial connected components $C_i$,  $C_i'$. To this end we simply use reversed BFS to order the vertices in each of these, one after another.
Consequently, if $r, t$ are the last two vertices in the ordering in a given such component, then $t$ is the root of the corresponding BFS tree, and hence $\{r, {   t}\}\in E(G')$. 

Consider a given $C_i'$, $1\leq i \leq T'$.
Since all the vertices in $C_i'$ are bad, 
in particular the last two vertices $r, t$  in $C_i'$ are bad vertices in $U'$, 
hence, 
$\{r,t\}$ is an active forward edge for $r$. 
Analogously, for a given $C_i$, $1\leq i\leq T$, all vertices in
$C_i$ are in $S\setminus U'$ and in particular so are the last two vertices $r, t$  in $C_i$. 
Moreover, since
all edges in $C_i \subset S$ are active, 
the edge $\{r, {   t}\}$ is an active forward edge for $r$. 
Additionally, by the definition of BFS, only the last vertex (i.e., the root $t$) can be a terminal vertex in any given $C_i'$ or $C_i$.

We are left to show that there is no terminal vertex in $W'$.  Since the minimum degree of $G'$ is at least $\sstar/2$ and $W'$ itself is an independent set {in $G'$}, each vertex in $W'$ has at least $\sstar/2$ incident edges joining it with 
vertices in 
$V(G')\setminus W'$, which come after $W'$ in the ordering. Since $W' \subset \Ub$, these forward edges are active. 
 Therefore vertices in {$W'$} cannot be terminal. 
 %
%
%
\end{proof}

\paragraph{Anchor sets. }
Let ${   \tb = \ceiling{{   48}ne^{-\dd^{2{   \alpha}}/{   4}}/\sstar}}, 
{   \tg = {   2}\ceiling{\frac{{   16}n}{\dd k'}\frac{1}{\tb}}\tb}$. 
{   In particular, $({2}\tb) \mid \tg$}.

Each vertex in $G'$ is either in $U'$ and thus is a bad vertex, or is in $S\setminus U'$ and thus could either be good or bad. 

Each 
bad vertex $v\in U'$ 
will be assigned 
in step C-2 
an anchor set $\ssetb_v$ of two elements 
of the following form $\{l+  (2\lambda) \tb k+ak, l+ (2\lambda+1) \tb k+ak\}$ where $\lambda \geq 0$ and $a\in [0,\tb-1]$ are integers to be determined 
 in Step C-2, and $l = 0$ or $1$ is 
the weight of $v$ modulo $k$ at the end of Step C-1 and thus is pre-determined. 
All the elements in $\ssetb_v$ are regarded {\textbf{modulo}} $\tg k$. 
For a fixed $l = 0$ or $1$, by varying $\lambda$ and $a$, these sets partition the set of integers $\{l+k\cdot \zz \}$ modulo $\tg k$. Different values of $a$ correspond to  different congruence classes modulo $\tb k$, denoted by $\overline{\mathcal{C}}_a(l) = \{l+ ak+ \tb k \cdot \zz \mod \tg k\}$. Note these are well defined as $\tg$ is divisible by $\tb$. 

Similarly,  {each 
vertex} $v\in V(G')\setminus U'$ will be assigned in step C-2 a set $\sset_v$ 
of the form 
$\{l+  (2\lambda) \tg k+ak, l+ (2\lambda+1) \tg k+ak\} \subset \zz_{\geq 0}$ where $\lambda, a$ are integers with $\lambda \geq 0$, $a\in [0,{   \tg}-1]$ to be determined in Step C-2, and $l = 2,3,4$, or $5$ is the vertex weight of $v$ modulo $k$ at the end of Step C-1. 
For a fixed $l$, these sets partition the set of non-negative integers $\{l+k\cdot \zz_{\geq 0} \}$. 
Different values of $a$ correspond to  different congruence classes, denoted by ${\mathcal{C}}_a(l) = \{l+ak+ \tg k \cdot \zz\}$.

\paragraph{Goal. }
For $i = 1, 2, \dots$, the algorithm sequentially analyzes each $v_i$, greedily modifying weights of active edges  incident to $v_i$ in $G'$  to prescribe the vertex weight of $v_i$ to an appropriately chosen 
set ${\ssetb_{v_i}}$ (if $v_i \in U'$) or 
${\sset_{v_i}}$ (if $v_i \notin U'$)  described above. Furthermore:
\begin{enumerate}[(i)] 
\item \label{item:goal1} In the process of analyzing a given 
$v_j$, for any $i<j$, the set $\sset_{v_i}$ or $\ssetb_{v_i}$ remains unchanged. 
That is,
the vertex weight of $v_i$  stays 
in $\sset_{v_i}$ if $v_i\notin U'$, and stays 
in $\ssetb_{v_i}$ modulo $\tg k$ if $v_i\in U'$ throughout this and all later stages of the algorithm. 
\item \label{item:goal2} The bad vertices in $U' = \Ub \cap U'$  have distinct $\ssetb_{v}$'s (modulo $\tg k$) assigned, with vertex weights being 0 or 1 modulo $k$.
\item \label{item:goal3}The bad vertices in $\Ub \setminus U' $  have distinct $\sset_{v}$'s assigned, with vertex weights being 4 or 5 modulo $k$.
\item \label{item:goal4}Each good vertex $v\in \Ug \setminus U' = \Ug$ has assigned a set $\sset_{v}$ different from the ones of vertices in $L(v)$, with vertex weights being 2 or 3 modulo $k$.
\end{enumerate}

\paragraph{Rules. }
Rules of modifying edge weights are as follows. Suppose we are analyzing $v_i$ and $\{u, v_i\}$ is an edge incident to $v_i$ in $G'$. 
\begin{enumerate}
\item For $v_1$, let $\sset_{v_1}$ (if $v_1 \notin U'$)  or $\ssetb_{v_1}$ (if $v_1 \in U'$) be the set of the desired form that contains the current vertex weight of $v_1$ (modulo $\tg k$, if $v_1\in U'$). Note the current weight of $v_1$ uniquely determines the set $\sset_{v_1}$ or $\ssetb_{v_1}$, respectively. We next move to $v_2$.
\item\label{item:R2} Only the active backward and forward edges of $v_i$ can get weights changed. 
\item\label{item:R3}  To modify an active forward edge $\{v_i, u\}$ of $v_i$, we admit adding to its weight 
an integer in $\{0,k, 2k,\dots \tb k\}$ if $v_i \in \Ub \cap U' = U'$ and respectively,  an integer in $\{0,k, 2k,\dots \tg k\}$ if $v_i \notin U'$. The forward edges of $v_i$ will together account for 
the congruence class $\overline{\mathcal{C}}_a(l)$, if $v_i \in U'$, or ${\mathcal{C}}_a(l)$, if $v_i \notin U'$, 
the weight of $v_i$ will ultimately belong to.

\item \label{item:R4} There are two 
options to modify each active backward edge $\{v_i, u\}$ of $v_i$ if both $v_i, u \in S\setminus U'$. We admit to modify the weight of $\{v_i,u\}$ by adding $0$ or $\tg k$ if the current weight of $u$ is the smaller value in $\sset_u$, and adding $0$ or $-\tg k$ if the current weight of $u$ is the larger value in $\sset_u$. 

\item\label{item:R5} There are two options to modify each active backward edge $\{v_i, u\}$ of $v_i$ if both $v_i, u \in U'$. We admit to modify the weight of $\{v_i,u\}$ by adding $0$ or $\tb k$ if the current weight of $u$ is congruent to 
the element $l + (2\lambda)\tb k + ak$ in $\ssetb_u$, and adding $0$ or $-\tb k$ if the current weight of $u$ is congruent to 
the element $l + (2\lambda+1)\tb k + ak$ in $\ssetb_u$  modulo $\tg k$.

\item\label{item:R6} If $v_i\in U'$ but $u \notin U'$, since all vertices in $U'$ come before vertices not in $U'$ (by Claim \ref{claim:ordering}), $\{v_i, u\}$ cannot be a backward edge of $v_i$. Follow Rule \ref{item:R3} to modify this edge as an active forward edge of $v_i$. 

\item \label{item:R7}If $v_i\notin U'$ but $u \in U'$,  there are two options to modify each active backward edge $\{v_i, u\}$. We admit to modify the weight of $\{v_i,u\}$ by adding $0$ or $\tg k$. 
Note however that if $v_i \in \Ug$ and $u \in U' \setminus S$, then  $\{v_i, u\}$ is not an active backward edge of $v_i$ by the definition of active edges for good vertices. 

\end{enumerate}
\begin{claim}\label{claim:goal1}
The Rules above guarentee Goal (\ref{item:goal1}) 
holds
throughout the algorithm.
\end{claim}
\begin{proof}
It is easy to see from Rules \ref{item:R4}-\ref{item:R6} that the updated weight of $u$ remains in $\sset_u$ or $\ssetb_u$, respectively, even after changing the weight of an active backward edge $\{v_i, u\}$ of $v_i$. Since 
$\ssetb_u$'s are regarded modulo $\tg k$,
Rule \ref{item:R7} also guarantees that the updated weight of $u$ stays 
in $\ssetb_u$.
\end{proof}

\begin{claim}\label{claim:StepCweightchange}
Assume (\ref{eq:sizeVh2}) holds for all  benchmarks $\bench{i} < \bench{j}$. 

After Step C-2,  each edge of $G[S]$ has weight in $[\ceiling{\ceiling{n/\dd}/2}-2\tg k-2, \ceiling{\ceiling{n/\dd}/2}+2\tg k+2]$ and each edge of $G'[B]$ has weight in $[\ceiling{\ceiling{n/\dd}/2}-2\tb k-2, \ceiling{\ceiling{n/\dd}/2}+2\tb k+2]$. 
During Step C, each edge across $U'\setminus S$ and $\Ug$ changes its weight by at most $\tb k + 2$ and   
each edge across $U' \setminus S$ and $S
\setminus \Ug$ changes its weight by at most $
\tg k + \tb k + 2$. Edges not in $G'$ do not get weights changed.

However, if $\Ub = \emptyset$, then $G' = G[S]$ and $U' =\emptyset$. After Step C-2,
each edge of $G[B]$ has then weight {   $1$} or $\ceiling{n/\dd}+a'$ and each edge in $G[S]$ has weight in $[\ceiling{\ceiling{n/\dd}/2}-2\tg k-2, \ceiling{\ceiling{n/\dd}/2}+2\tg k+2]$. Moreover, during Step C, no weights are changed for edges across $B$ and $S$. 
\end{claim}
\begin{proof}
Each edge can be modified at most twice in Step C-2: once as an active forward edge and once as an active backward edge. By the Rules above, 
each time its weight is modified, it can be changed by at most $\tg k$, thus each edge weight can be changed in total by at most $2\tg k$ in Step C-2. Step C-1 changes in turn an edge 
weight by at most 2 by Claim \ref{claim:stepC1}. Since at the beginning of Step C-1, each edge weight in $E(S)$   
 was initialized as $\ceiling{\ceiling{n/\dd}/2}$,  the result for edges in $E(S)$ follows.

For each edge in $E(G') \cap E(B)$, both its ends 
are bad vertices, 
and thus Step C-1 changes its weight by at most 2 and Step C-2 changes its weight by at most $2\tb k$ by Rules \ref{item:R3} and \ref{item:R5}. The result follows, since the edge weight  
was initialized as $\ceiling{\ceiling{n/\dd}/2}$ at the beginning of Step C-1. 
 
 For each edge between $v\in \Ug$ and $u\in {  U'}\setminus S$, 
 again Step C-1 changes its weight by at most 2. Vertex $u$ must be a bad vertex and it comes before $v$ in the ordering. In Step C-2, while analyzing 
 $u$,  the weight of $\{v, u\}$ is changed by at most $\tb k$ by {  Rule \ref{item:R3}}. 
In the process of analyzing 
$v$ in turn, since $v \in \Ug$ but $u \notin S$, the edge $\{v, u\}$ is not active for $v$ by the definition of active edges for good vertices, and thus could not be changed. 
 The result follows. 
 
 For each edge between $v\in S\setminus \Ug$ and $u\in {  U'}\setminus S$, 
 $v$ is either in $U'$ or not.
If $v \in U'$, then since $u\in U'$, Step C-2 changes the weight of {  $\{u, v\}$} by at most $2\tb k$ by Rules \ref{item:R3} and \ref{item:R5}.   If $v \notin U'$, then $v$ comes after $u$ in the ordering. In Step C-2, the weight of $\{u,v\}$, being an active forward edge of $u$,  is thus changed by at most $\tb k$, 
and,  as an active backward edge of $v$, changed by at most $\tg k$, due to Rules \ref{item:R3} and \ref{item:R7}. The result follows analogously as above, as $\tg > \tb$. 

The case when $\Ub = \emptyset$ follows from {  Corollary} \ref{lem:step2edgew} by noting that $V(G') = S$ and $U'\subset \Ub =\emptyset$, by Claim \ref{claim:G'}, and thus Step C does not change weights of edges in $E(B)$ or edges across $B$ and $S$. 
\end{proof}




\begin{lemma}\label{lem:disBandLv}
Suppose all the inequalities in Corollary \ref{cor:sizeUbLv} hold. 
There is a way to modify the edge weights 
abiding the Rules of the algorithm so that Goals (\ref{item:goal1})-(\ref{item:goal4}) are fulfilled. 
\end{lemma}
\begin{proof}
Goal (\ref{item:goal1}) is fulfilled 
by Claim \ref{claim:goal1}. 
By the Rules, all edge weights in $G'$ are changed in Step C-2 by {multiples of $k$}. Thus  the modulo $k$ conditions in Goals (\ref{item:goal2}) to (\ref{item:goal4}) automatically hold by 
the preparatory measures from
Step C-1 (Claim \ref{claim:stepC1}). It remains to show 
that we can process the vertices $v_i$  for $i = 1,2,\dots$ complying with the Rules so that the rest of the 
conditions in Goals (\ref{item:goal2}) to (\ref{item:goal4}) hold.

Suppose we are analyzing a given 
$v_i$ which is not a terminal vertex 
nor a vertex immediately preceding 
a terminal vertex. 
Suppose at the end of Step C-1 the weight of $v_i$ equals $l_i$ modulo $k$. 
 
 Case 1. Suppose $v_i \in U'$.
We first choose any of its active forward edges, say $e$.
Due to adding to 
its weight 
admissible values in the set $\{0, k, 2k, \dots, \tb k\}$ (Rule \ref{item:R3}), the weight of $v_i$ runs through all the congruence classes $\overline{\mathcal{C}}_a(l_i)$ with $0 \leq a \leq \tb-1$. 
The 
sets $\ssetb_{u}$ fixed already for $u \in  U'$ prior to $v_i$ occupy at most $|U'| \leq |\Ub|$ congruence classes $\overline{\mathcal{C}}_a(l_i)$, with possible duplicates. By an averaging argument, there must be an $a^*$ such that at most
$|\Ub|/\tb $ of these 
sets $\ssetb_u$ 
are in the same congruence class $\overline{\mathcal{C}}_{a^*}(l_i)$. Fix such a congruence class  $\overline{\mathcal{C}}_{a^*}(l_i)$ and assure the weight of $v_i$ belongs in it
via adjusting the weight of $e$. 
We then modify the rest of the active forward edges of $v_i$ by adding $0$ or $\tb k$ to their weights, and  modify the weights of its active backward edges by $0$ or $\pm \tb k$ according to Rules \ref{item:R5} and \ref{item:R6}. Since $v_i$ is incident to at least $\sstar/2$ active edges in $G'$ (by Claim \ref{claim:G'}), the vertex weight of $v_i$ can thus be attributed 
at least
$\min(\tg / \tb, {  \sstar/2})$ 
consecutive  terms in the set $\{l_i+a^*k + \tb k\cdot \zz \mod \tg k\}=\overline{\mathcal{C}}_{a^*}(l_i)$.  
Since each of the at most $|\Ub|/\tb$ existing sets $\ssetb_u$ in $\overline{\mathcal{C}}_{a^*}(l_i)$ blocks two consecutive terms in $\overline{\mathcal{C}}_{a^*}(l_i)$, 
we can find an achievable ${  \overline{\sset}}_{v_i} \subset \overline{\mathcal{C}}_{a^*}(l_i)$ that is disjoint from all the prior $\ssetb_u$ (modulo $\tg k$) with $u \in {  U'}$ 
if 
\[ {    \min(\tg / \tb,}  {  \sstar/2}  {   ) >}  {   2|\Ub|/\tb. } 
\]
Since we assumed inequalities in Corollary \ref{cor:sizeUbLv} hold, $ |\Ub| \leq {   3} {  n \exp(-\delta^{2{   \alpha}}/{   4})}$, and thus: 
\begin{align} 
 {    
 \min(\tg/\tb,}  {  \sstar/2} {   ) > 4|\Ub|/\tb+ 2.}
\label{eq:choice}
\end{align}
Therefore we are done with $v_i$.

Case 2. Suppose $v_i \notin U'$.
The analysis is almost the same as in Case 1. 
Suppose $v_i$ is a good vertex.
By an averaging argument, there must be an $a^*$ such that the congruence class ${\mathcal{C}}_{a^*}(l_i)$ hosts at most
$|L(v_i)|/\tg $ prior sets $\sset_u$ with $u \in L(v_i)$. We insert the weight of $v_i$ in this congruence class via 
modifying  one of its active forward edges by adding to its weight one of the admissible integers in $\{0, k, 2k, \dots, \tg k\}$.
We then modify the weights of the rest of the active edges of $v_i$ by an integer in $\{0, \pm \tg k\}$ abiding the Rules. Since $v_i$ is incident to at least $\sstar/2$ active edges in $G'$ (by Claim \ref{claim:G'}), the vertex weight of $v_i$ can thereby be attributed at least 
{  $\sstar/2$} 
consecutive  terms in ${\mathcal{C}}_{a^*}(l_i) = \{l_i+a^*k + \tg k\cdot \zz \}$.  
Since each of the at most $|L(v_i)|/\tg$ existing sets $\sset_u$ in ${\mathcal{C}}_{a^*}(l_i)$
and with $u\in L(v_i)$
 blocks  two consecutive terms
 in ${\mathcal{C}}_{a^*}(l_i)$, 
we can find an achievable $\sset_{v_i} \subset {\mathcal{C}}_{a^*}(l_i)$ that is disjoint from all the prior  
{  $\sset_u$  with $u \in L(v_i)$} 
if 
\begin{equation}
{   \sstar/2} {    > } {   
2|L(v_i)|/\tg. } \nonumber
\end{equation} 
Thus we are done with $v_i$ by the bound $|L(v)| \leq 2{  \ceiling{\sstar/k'} (n/\dd)}$ 
in Corollary~\ref{cor:sizeUbLv}, which holds for every $v\in S$, and  implies in particular that: 
\begin{equation}
{   \sstar/2} {    > } {   
4}{   |L(v)|/\tg + 2. } \label{eq:choice2}
\end{equation} 
The case when $v_i \notin U'$ and $v_i \in \Ub$ follows by the same argument, with $L(v_i)$ replaced by $\Ub \setminus U'$. Thus we are done with $v_i$ if only
\[ {   \sstar/2} {   >} {  2|\Ub|/\tg} {   \geq} {   2|\Ub \setminus U'|/\tg. }
\]
The first inequality above follows however by (\ref{eq:choice}), as $\tg> \tb$, while the second one trivially holds.

We are left to show how to manage $r_i, t_i$ where $t_i$ is a terminal vertex and $r_i$ is the vertex preceding $t_i$ in the ordering. We apply a similar approach as above, analyzing the both vertices simultaneously. By Claim \ref{claim:ordering}, $\{ r_i, t_i\}$ is an edge in $G'$, which is an active forward edge for $r_i$. 

Case 1'. Suppose both $r_i, t_i$ are in $U'$. Thus both vertices are bad vertices. By an averaging argument,  
when we reach 
$r_i$, we can modify its forward edge $\{r_i, t_i\}$ so that the two new congruence classes of $r_i, t_i$ each hosts at most $2|U'|/\tb \leq 2|\Ub|/\tb$ prior sets $\ssetb_u$ with $u\in U'$, ignoring temporarily $r_i$ from the point of view of $t_i$. Fix such two new congruence classes for $r_i, t_i$ by choosing an admissible modification for 
the edge $\{r_i, t_i\}$. By the same argument as before, since
$
{    \min(\tg/\tb,} {  \sstar/2}{   ) >
4|\Ub|/\tb+2}$ 
{(  where ``$+2$'' might be necessary to adjust the weight of $t_i$ in the case when it belongs to the same congruence class as $r_i$)}, 
which holds by (\ref{eq:choice}),
via changing the weights of active backward edges of $r_i$ and all its other active forward edges except $\{r_i, t_i\}$ by values in $\{0, \pm \tb k\}$ complying with the Rules, there is an achievable choice of $\ssetb_{r_i}$ that is disjoint from all the sets $\ssetb_u$ for $u\in U'$ 
prior to $r_i$. 
Next
via changing the weights of the active backward edges of $t_i$ except $\{r_i, t_i\}$ by values in $\{0, \pm \tb k\}$ complying with the Rules, and by (\ref{eq:choice}) again, we can find an achievable {  $\ssetb_{t_i}$} that is disjoint from all the prior sets {  $\ssetb_u$} with  $u\in U'$ including {  $\ssetb_{r_i}$}.

Case 2'. Suppose neither $r_i$ nor $t_i$ are in $U'$, i.e., $r_i, t_i \in S\setminus U'$. If both $r_i, t_i$ are bad vertices,  we carry out the same reasoning as in Case 1' with $\tb$ replaced by $\tg$.  The inequality we need to guarantee in such a case 
is
${   \sstar/2}{    > 
4|\Ub|/\tg+2,}
$ which holds by~(\ref{eq:choice}).
If both $r_i$ and $t_i$ are good vertices, again we use the same reasoning as in Case 1' with $|\Ub|$ replaced by $\max(|L(r_i)|, |L(t_i)|)$ and {  $\tb$ replaced by $\tg$}. 
The inequality we need to guarantee this time is: 
\[{  
\sstar/2} {    > 
4\max(|L(r_i)|, |L(t_i)|)/\tg+2},
\]
which holds by~(\ref{eq:choice2}). 
If finally one of $r_i, t_i$ is good and the other one is bad, 
say among $\{r_i, t_i\}$, 
$u$ is the bad vertex and $v$ is the good one, by modifying the active forward edge $\{r_i, t_i\}$, we may assure 
each of the weights of $t_i$ and $r_i$ to be in a congruence class containing at most $(|L(v)|+|\Ub\setminus U'|) /\tg \leq (|L(v)|+|\Ub|)/\tg$ prior sets $\sset_{v'}$ with $v' \in L(v)$ (in the case of $v$) and $\sset_{u'}$ with $u' \in \Ub\setminus U'$ 
(in the case of $u$). 
Fix such a congruence class by choosing an appropriate admissible weight for $\{r_i, t_i\}$. By changing the weights of the remaining active edges of $v$ by values in $\{0, \pm \tg k\}$ complying with the Rules, we can find an achievable $\sset_v$ that is disjoint from all $\sset_{v'}$ with $v' \in L(v)$ prior to $v$ if 
\[
{  \sstar/2} > {  2(|L(v)|+|\Ub|)/\tg}.
\]
The inequality holds by~ (\ref{eq:choice}) and~(\ref{eq:choice2}), 
and thus $v$ can be successfully processed. 
To adjust the weight of 
the bad vertex $u$, we analogously as above change the weights of active edges of $u$ except the edge $\{u,v\} = \{r_i, t_i\}$ by values in $\{0, \pm \tg k\}$ complying with the Rules. 
By the same argument as for $v$, we can find a achievable set  $\sset_u$ 
that is disjoint from the ones of the other prior bad vertices not in $U'$ if 
${  \sstar/2} > {  2}(|L(v)|+|\Ub|)/\tg,$ which again holds by~ (\ref{eq:choice}) and~(\ref{eq:choice2}). 
This finishes the proof by the third condition in Claim~\ref{claim:ordering}.
\end{proof}

\subsection{Proof of Theorems \ref{thm:main2} and \ref{thm:main1}}
\begin{lemma}\label{lem:Lvnoequal}
Assume 
{   all statements in Corollary~\ref{cor:sizeUbLv} hold}.
Then for any two good vertices $u, v$ with $u \notin L(v)$, the weights of $v$ and $u$ are distinct after Step C
provided that $\dd$ and $n/\dd$ are sufficiently large. 
\end{lemma}
\begin{proof}
Suppose $v \in S_q{   \cap \Ug}$ for some $1\leq q \leq k'$. By Corollary \ref{lem:step2edgew}, prior to Step C, the weight of any edge $e$ between $v$ and its neighbor in $B$ is in the interval \[ \left[\ceiling{\ceiling{\frac{n}{\dd}}\frac{1}{3k'}}(k'+q) - \left(\frac{{   11}n}{\dd^{1+\epsilon}}+2\right),\ \ceiling{\ceiling{\frac{n}{\dd}}\frac{1}{3k'}}(k'+q)+ \left(\frac{{   11}n}{\dd^{1+\epsilon}}+2\right)\right].
\] 
Let $\x = 1$ if $\Ub \neq \emptyset$ and $\x = 0$ if $\Ub = \emptyset$.  By Claim~\ref{claim:StepCweightchange}, 
since $v\in \Ug$, the weight of an edge $e$ between $v$ and $B$ is changed by at most $\x (\tb k +2)$ during Step C. 
Thus  the final weight of $e$ is in the interval
\[ \left[\ceiling{\ceiling{\frac{n}{\dd}}\frac{1}{3k'}}(k'+q) - \left(\frac{{   11}n}{\dd^{1+\epsilon}} + \x \tb k +2\x+2\right),\ \ceiling{\ceiling{\frac{n}{\dd}}\frac{1}{3k'}}(k'+q)+ \left(\frac{{   11}n}{\dd^{1+\epsilon}}+ \x\tb k +2\x+2\right)\right].
\] 
Since  ${   \tb = \ceiling{{   48}ne^{-\dd^{2{   \alpha}}/{   4}}/\sstar}}$, 
if $\dd \leq  \sqrt{n}$ and $\dd$ is sufficiently large, then $n/\dd^{1+\epsilon} >\tb k$. 
If $\dd > \sqrt{n}$ in turn, then 
{   by the last statement in Corollary~\ref{cor:sizeUbLv}, $\Ub = \emptyset$ and hence}
$\x=0$.  
Thus the weight of $e$ is in any case in the interval
\[ \left[\ceiling{\ceiling{\frac{n}{\dd}}\frac{1}{3k'}}(k'+q) - \left(\frac{{   12}n}{\dd^{1+\epsilon}} + 4 \right),\ \ceiling{\ceiling{\frac{n}{\dd}}\frac{1}{3k'}}(k'+q)+ \left(\frac{{   12}n}{\dd^{1+\epsilon}}+4 \right)\right].
\] 
Therefore, the weight of $v$ coming from its neighbors in $B$ equals at least
\begin{align}
& \left(\ceiling{\ceiling{\frac{n}{\dd}}\frac{1}{3k'}}(k'+q) - \left(\frac{{   12}n}{\dd^{1+\epsilon}} + 4 \right)\right)\deg_B(v).  \label{eq:vgoodBlow} 
\end{align}
Similarly, the weight of $v$ coming from its neighbors in $B$ equals at most
\begin{align}
& \left(\ceiling{\ceiling{\frac{n}{\dd}}\frac{1}{3k'}}(k'+q)  + \left(\frac{{   12}n}{\dd^{1+\epsilon}} + 4 \right)\right)\deg_B(v).  \label{eq:vgoodBup}
\end{align}
{   Note that by definition, for large enough $\delta$ and $n/\delta$, $k < k'/960 < 
10^{-5}n/\delta$, and thus 
\begin{align}
& \tg k < \frac{32 n}{\delta k'}k+2\tb k < \frac{1}{30}\frac{n}{\delta} +\frac{96ne^{-\delta^{2\alpha}/4}k}{\delta^{1/2+\epsilon+\alpha}} +2 k < \frac{1}{20}\frac{n}{\delta} . \label{kk'Bounds}
\end{align}
Therefore, by~(\ref{kk'Bounds})} and {   Claim~\ref{claim:StepCweightchange}}, 
weights of edges in $E(S)$ are  {   contained in the interval 
\begin{equation}
\left[\ceiling{\ceiling{\frac{n}{\dd}}\frac{1}{2}}-2\tg k-2, \ceiling{\ceiling{\frac{n}{\dd}}\frac{1}{2}}+2\tg k+2\right] \subset [1, {   n/\dd}). 
\label{eq:esinterval}
\end{equation}}
  Since $v$ is a good vertex, by the definition of $\ysq{q}$, 
\begin{equation}
 |\deg_B(v) - {   \deg(v)(\dd-\sstar)}/\dd| {   \leq} \deg(v)^{1/2+\epsilon}, 
 \label{eq:degS}
\end{equation}  
and thus,
\begin{equation}
 |\deg_S(v) - \deg(v)\sstar/\dd| {   \leq} \deg(v)^{1/2+\epsilon}. \label{eq:degS2}
\end{equation}  

By 
(\ref{eq:vgoodBlow}), (\ref{eq:esinterval}), and (\ref{eq:degS}), we  conclude that for $\delta$ and $n/\delta$ large enough,  the weight of $v$ equals at least
\begin{align}
&  \left(\ceiling{\ceiling{\frac{n}{\dd}}\frac{1}{3k'}}(k'+q)  - \left(\frac{{   12}n}{\dd^{1+\epsilon}} + 4 \right)\right)\deg_B(v) + 0 \cdot \deg_S(v) \nonumber \\
 {   \geq} & \left(\ceiling{\ceiling{\frac{n}{\dd}}\frac{1}{3k'}}(k'+q)  - \left(\frac{{   12}n}{\dd^{1+\epsilon}} + 4 \right)\right)\left(\frac{\deg(v)(\dd-\sstar)}{\dd} - \deg(v)^{1/2+\epsilon}\right) \nonumber \\
\geq & \deg(v)\left( 1- \frac{\sstar}{\dd}- \frac{1}{\dd^{1/2-\epsilon}}  \right)\left(\ceiling{\ceiling{\frac{n}{\dd}}\frac{1}{3k'}}(k'+q)  - \left(\frac{{   12}n}{\dd^{1+\epsilon}} + 4 \right)\right)\nonumber \\
> & \deg(v)\left( 1- \frac{2\sstar}{\dd} \right)\left(\ceiling{\ceiling{\frac{n}{\dd}}\frac{1}{3k'}}(k'+q)  - \left(\frac{{   12}n}{\dd^{1+\epsilon}} +4 \right)\right) \nonumber\\
> & \deg(v)\left(\ceiling{\ceiling{\frac{n}{\dd}}\frac{1}{3k'}}(k'+q)  - \left(\frac{{   12}n}{\dd^{1+\epsilon}} + 4 \right) - \frac{2\sstar}{\dd}\frac{n}{\dd}\right)\nonumber \\
> &  \deg(v)\left(\ceiling{\ceiling{\frac{n}{\dd}}\frac{1}{3k'}}(k'+q)  - \left(\frac{{   13}n}{\dd^{1+\epsilon}} + 4 \right)\right), \label{eq:vSqlow}
\end{align}
where the last inequality follows by~(\ref{EpsilonAlpha}).
Similarly, by (\ref{eq:vgoodBup}), (\ref{eq:esinterval}) 
and~(\ref{eq:degS2}), when $\dd$ and $n/\dd$ are sufficiently large, the weight of $v$ equals at most
\begin{align}
& \left(\ceiling{\ceiling{\frac{n}{\dd}}\frac{1}{3k'}}(k'+q)  + \left(\frac{{   12}n}{\dd^{1+\epsilon}} + 4 \right)\right)\deg_B(v) + \deg_S(v){   \frac{n}{\dd}}\nonumber \\
\leq & \left(\ceiling{\ceiling{\frac{n}{\dd}}\frac{1}{3k'}}(k'+q)  + \left(\frac{{   12}n}{\dd^{1+\epsilon}} + 4 \right)\right)\deg(v) + \left(\deg(v)\frac{\sstar}{\dd}+\deg(v)^{1/2+\epsilon}\right){   \frac{n}{\dd}}\nonumber\\
\leq & \left(\ceiling{\ceiling{\frac{n}{\dd}}\frac{1}{3k'}}(k'+q)  + \left(\frac{{   12}n}{\dd^{1+\epsilon}} + 4 \right)\right)\deg(v) + \deg(v)\left(\frac{\sstar}{\dd} + \frac{1}{\dd^{1/2-\epsilon}}\right){   \frac{n}{\dd}}\nonumber \\
<& \left(\ceiling{\ceiling{\frac{n}{\dd}}\frac{1}{3k'}}(k'+q)  + \left(\frac{{   12}n}{\dd^{1+\epsilon}} + 4 \right)  
+ \frac{2\sstar}{\dd}{   \frac{n}{\dd}}
\right)\deg(v)\nonumber \\
< &   \left(\ceiling{\ceiling{\frac{n}{\dd}}\frac{1}{3k'}}(k'+q)  + \left(\frac{{   13}n}{\dd^{1+\epsilon}} + 4 \right)  
\right)\deg(v), \label{eq:vSqup}
\end{align}
where the last inequality again follows by~(\ref{EpsilonAlpha}).
Therefore (\ref{eq:vSqlow}) and (\ref{eq:vSqup}) provide the lower 
and upper bounds on the weight of $v\in S_q \cap \Ug$ after Step C. Analogous bounds hold also by the same reasoning for any $u\in S_p \cap \Ug$, where  $1\leq p, q \leq k'$. 
Thus if the weights of $v$ and $u$ are equal 
after Step C, both of the following two inequalities must hold: 
\begin{align*}
\left(\ceiling{\ceiling{\frac{n}{\dd}}\frac{1}{3k'}}(k'+q)  + \left(\frac{{   13}n}{\dd^{1+\epsilon}} + 4 \right)  
\right)\deg(v) > 
\left(\ceiling{\ceiling{\frac{n}{\dd}}\frac{1}{3k'}}(k'+p)  - \left(\frac{{   13}n}{\dd^{1+\epsilon}} + 4\right)  
\right)\deg(u); \\
\left(\ceiling{\ceiling{\frac{n}{\dd}}\frac{1}{3k'}}(k'+q)  - \left(\frac{{   13}n}{\dd^{1+\epsilon}} + 4 \right)  
\right)\deg(v) <
\left(\ceiling{\ceiling{\frac{n}{\dd}}\frac{1}{3k'}}(k'+p)  + \left(\frac{{   13}n}{\dd^{1+\epsilon}} + 4 \right)  
\right)\deg(u).
\end{align*}
These two conditions are equivalent to $u\in L(v)$ and $v\in L(u)$, cf. Definition~\ref{def:uvclose}. Therefore if $u\notin L(v)$, then the weights of $u$ and $v$ cannot be the same after Step C. 
\end{proof}

We are finally ready to argument that Theorem~\ref{thm:main2} indeed holds.

If  $\dd$ or $n/\dd$ are small, say $\dd < {   c}$ or $n/\dd < {   c}$ for some absolute constant ${   c}$, we make use of the result in ~\cite{KKP}, which implies that $s(G) \leq (n/\dd)(1+ 5) + 6$, and therefore 
$s(G) \leq \frac{n}{\dd} + 5{   c} + 6$ in the case when $n/\dd <{   c}$, and $s(G) \leq \frac{n}{\dd}\left(1 + \frac{5{   c}^{\epsilon}}{\dd^{\epsilon}}\right) + 6$ in the case when $\dd <{   c}$.

From now on we can assume that $\dd$ and $n/\dd$ are sufficiently large. 
Thus with positive probability, all the inequalities in Corollary~\ref{cor:sizeUbLv} hold. In particular, $\Ub = \emptyset$ if $\dd > \sqrt{n}$. 
We first show that all the vertex weights are distinct after Step C. 
Most vertices in $B$ receive distinct weights in $\zz'$ within Step B, cf. Corollary~\ref{lem:step2edgew}. The remaining ones are distinguished in Step C by means of weights outside $\zz'$.
By Lemma \ref{lem:disBandLv}, Goals (\ref{item:goal1}) to (\ref{item:goal4}) can be achieved in Step C. By the Goals, it is clear that all the vertex weights are distinct, with the only possible exception 
between  vertices $v \in \Ug$ and good vertices not in $L(v)$. However, Lemma \ref{lem:Lvnoequal} shows that if $u \notin L(v)$, then the weights of $u,v$ cannot be identical. Thus we have shown that all vertices in $G$ indeed have distinct weights.  
 
We next bound the values of the final edge weights 
after Step C. 
By Claim \ref{claim:StepCweightchange}, the fact that $\tg {   >} \tb$, and (\ref{eq:esinterval}), weights of edges in $G[S]$ and $G'[B]$ are contained in $[1, {   n/\dd})$. Other edges in $B$ but not in $G'$ do not get weights changed during Step C  (by Claim~\ref{claim:StepCweightchange}). Thus by
Corollary \ref{lem:step2edgew}, these edges in $B$ have weights either $1$ or $\ceiling{n/\dd}+a'$. As for the edges across $B$ and $S$, prior to Step C, by (\ref{eq:f1f2intBS}) in Corollary~\ref{lem:step2edgew}, their weights were in the interval 
 $ \left[\ceiling{\ceiling{\frac{n}{\dd}}\frac{1}{3k'}}k'- \left(\frac{{   11}n}{\dd^{1+\epsilon}}+2\right),\ceiling{\ceiling{\frac{n}{\dd}}\frac{1}{3k'}}2k'+ \left(\frac{{   11}n}{\dd^{1+\epsilon}}+2\right)\right] \subset \left[\frac{n}{{   4}\dd}, \frac{3n}{4\dd} \right]$ after Step B. During Step C, by Claim \ref{claim:StepCweightchange} again,  their weights are changed by at most $
\tg k + \tb k + 2 {    < \frac{n}{4\delta}}$, by~(\ref{kk'Bounds}).
Thus after Step C, the weights between $B$ and {   $S$} lie in the interval
$
[1, n/\dd).
$
Thus we have shown that all final edge weights are positive, with the maximum weight at most  \[ \ceiling{n/\dd}+a' \leq \ceiling{\frac{n}{\dd}}+ \ceiling{\frac{{   7} n}{{   \dd}k}}.\] 
By our choice of $k$, Theorem \ref{thm:main2} holds, i.e. there are absolute constants $C_1, C_2$ (for any fixed $\epsilon \in (0, 0.25)$) such that $s(G)\leq \frac{n}{\delta}+\frac{C_1n}{\delta^{1+\epsilon}}+C_2$. To see why Theorem \ref{thm:main1} holds, note that when $\dd^{1+\epsilon} {   \geq} n$, then 
$\frac{C_1n}{\delta^{1+\epsilon}}$  is upper bounded by $C_1$.


\section{Conclusion}
In this paper, we proved a uniform upper bound $s(G) \leq \frac{n}{\dd}(1+C_1/\dd^{\gamma}) + C_2$ where $C_1, C_2$ are absolute constants for any $\gamma \in (0, 0.25)$.  This confirms the Faudree-Lehel Conjecture for $\dd \geq n^{1/(1-\gamma)}$.  
We did not strive to optimize all the constants in our result. We believe that with a slightly modified construction one should be in particular able to magnify the 0.25 upper bound on $\gamma$. Our bound matches the bound in Conjecture \ref{conj:main2} asymptotically when $\delta$ is large. 
It would also be interesting to prove a bound of the form $s(G) \leq \frac{n}{\dd}(1+o_n(1)) + C$ for some absolute constant $C$.


\bibliographystyle{plain}
\bibliography{sample3}

\end{document}